\documentclass{scrartcl}
\usepackage{amssymb}
\usepackage{amsfonts}
\usepackage{amsmath}
\usepackage{pgf}
\usepackage{hyperref}
\usepackage{enumerate}
\newcommand{\bbbc}{\mathbb{C}}
\newcommand{\bbbn}{\mathbb{N}}
\newcommand{\bbbr}{\mathbb{R}}
\newcommand{\Idx}{\mathcal{I}}
\newcommand{\ctI}{\mathcal{T}_{\Idx}}
\newcommand{\ctIl}[1]{\mathcal{T}_{\Idx}^{(#1)}}
\newcommand{\ctII}{\mathcal{T}_{\Idx\times\Idx}}
\newcommand{\lfII}{\mathcal{L}_{\Idx\times\Idx}}
\newcommand{\lfaII}{\mathcal{L}^+_{\Idx\times\Idx}}
\newcommand{\lfiII}{\mathcal{L}^-_{\Idx\times\Idx}}
\newcommand{\supp}{\mathop{\operatorname{supp}}\nolimits}
\newcommand{\dist}{\mathop{\operatorname{dist}}\nolimits}
\newcommand{\diam}{\mathop{\operatorname{diam}}\nolimits}
\newcommand{\sons}{\mathop{\operatorname{sons}}\nolimits}
\newcommand{\sd}[1]{\mathop{\operatorname{sd}}_{#1}\nolimits}
\newcommand{\level}{\mathop{\operatorname{level}}\nolimits}
\newcommand{\bi}{{\mathbf i}}
\newcommand{\punkt}{\boldsymbol{\cdot}}
\newcommand{\enorm}{\mathfrak{n}_{dp}}
\newcommand{\Cin}{C_\mathrm{in}}
\newcommand{\Cmi}{C_\mathrm{mi}}

\newtheorem{theorem}{Theorem}
\newtheorem{lemma}[theorem]{Lemma}
\newtheorem{corollary}[theorem]{Corollary}
\newtheorem{definition}[theorem]{Definition}
\newtheorem{remark}[theorem]{Remark}
\newtheorem{assumption}[theorem]{Assumption}

\newenvironment{proof}{\emph{Proof.}}{\hfill$\Box$}
\allowdisplaybreaks
\numberwithin{equation}{section}
\numberwithin{theorem}{section}
\title{Approximation of the high-frequency Helmholtz kernel
       by nested directional interpolation}
\author{Steffen B\"orm and Jens M. Melenk}
\date{\today}
\begin{document}
\maketitle

\begin{abstract}
We present and analyze an approximation scheme for a class of highly oscillatory
kernel functions, taking the 2D and 3D Helmholtz kernels as examples.
The scheme is based on polynomial interpolation combined
with suitable pre- and postmultiplication by plane waves.
It is shown to converge exponentially in the polynomial degree
and supports multilevel approximation techniques. 
Our convergence analysis 
may be employed to establish exponential convergence of certain classes of 
fast methods for discretizations of the Helmholtz integral operator that
feature polylogarithmic-linear complexity.  
\end{abstract}

%
%
\section{Introduction}
\label{sec:intro}
Integral operators with highly oscillatory kernels arise, for example, 
in time-harmonic settings of acoustic and electromagnetic scattering. 
They have the form 
\begin{equation}
\label{eq:integral-operator}
  u \mapsto \mathcal{G}[u](x) := \int_\Gamma k(x,y) u(y) \,dy,
\end{equation}
where $\Gamma$ is typically a subset of $\bbbr^n$ or a $(n-1)$-dimensional
submanifold of $\bbbr^n$.
Prominent examples of kernels $k$ that are studied in detail in the
present work are the three- and two-dimensional Helmholtz kernels 
\begin{align}\label{eq:3D_2D_helmholtz_kernels}
  g(x,y)
  &= \frac{\exp(\bi \kappa \|x-y\|)}{4\pi \|x-y\|}, &
  h(x,y)
  &= \frac{\bi H^{(1)}_0(\kappa \|x-y\|)}{4},  
\end{align} 
where $H^{(1)}_0$ is the first kind Hankel function of order $0$. For 
real, large \emph{wave numbers} $\kappa \in \bbbr_{\geq 0}$, these two functions 
are highly oscillatory.  A more general setting is one where the factorization
\begin{equation}\label{eq:factorization-of-g}
  k(x,y) = \Bigl[ k(x,y) \exp(-\bi \kappa \|x - y\|)\Bigr]
           \exp(\bi \kappa \|x - y\|) 
\end{equation}
decomposes the kernel $k$ into a highly oscillatory 
contribution $\exp(\bi \kappa \|x - y\|)$ and a non-oscillatory (or at
least less oscillatory) term $k(x,y) \exp(-\bi \kappa \|x - y\|)$.
Again, the two- and three-dimensional Helmholtz kernels are precisely
of this type. 

Discretization of the integral operator (\ref{eq:integral-operator}) by 
Galerkin or collocation methods leads to fully populated system matrices.
In order to reduce the complexity of these methods, various 
compression schemes have been devised in the past.
Techniques that address the particular challenges arising in the
high-frequency setting of large wave numbers $\kappa$ include the
\emph{fast multipole method}, \cite{RO93,GRHUROWA98,DA00,chen-et-al06},
the \emph{butterfly method}, \cite{BOMI96,CADEYI09,LIYAYI15}, 
and \emph{directional approximations}, \cite{BR91,ENYI07,MESCDA12,BEKUVE15}.  
In many of these approaches, the underlying justification is based on
approximating the kernel using suitable expansion systems; in the high
frequency setting, these expansion systems are typically not polynomial
in order to account for the oscillatory behavior of $k$. 
Polylogarithmic-linear complexity of the algorithms requires a second
ingredient: a multilevel structure.
Suitable expansion systems are given on each level and a fast transfer
between levels has to be effected, e.g., by ``re-expansion''.

From the point of view of numerical analysis, the stability of such an
iterated approximation requires investigation.
For a class of expansion systems that consist of products of plane waves
and polynomials, we present a full analysis of the approximation properties
and a stability analysis of the corresponding iterated re-expansion.
We stress that underlying our approximations is polynomial
interpolation (e.g., Chebyshev interpolation).
It is this very powerful tool that permits a full analysis of the 
algorithm of \cite{MESCDA12}, which we present here.
Additionally, the techniques developed here can be used to 
analyze other re-expansion processes that are based on polynomial interpolation, 
for example, the butterfly algorithm proposed in \cite[Sec.~4]{CADEYI09} and 
\cite{kunis-melzer12}. 
This claim to possible generalizations is substantiated in \cite{BOBOME17},  
where the 
analysis given in \cite{demanet-ferrara-maxwell-poulson-ying10} of the 
particular butterfly algorithm of \cite[Sec.~4]{CADEYI09} is sharpened
using the tools of the present paper.

It is worth mentioning that many algorithms in the literature such as 
\cite{BEKUVE15} enforce the required multilevel structure by algebraic means; 
while these algorithms can be very successful in practice, a complete
analysis is still missing since similarly powerful analytical tools
are not available. 
We defer a more detailed discussion of our problem setting to
Section~\ref{sec:ideas} after the necessary notation has been provided. 

\bigskip
\emph{Concerning notation:} 
We write $\langle \cdot,\cdot\rangle$ for the \emph{bilinear} form 
\begin{align}
\label{eq:euclidean-innerproduct}
\bbbc \times \bbbc &\rightarrow \bbbc, & (x,y) \mapsto   \langle x, y \rangle
  &:= \sum_{j=1}^n x_j y_j. 
\end{align}
This is \emph{not} a sesquilinear form, but the restriction to
the real subspace $\bbbr^n$ is the standard Euclidean inner
product. Throughout, $\|\cdot\|$ denotes the Euclidean norm on $\bbbr^n$. 
For vectors $c \in \bbbr^n$ and functions $u$, we use the shorthand 
$\exp(\bi \kappa \langle c,\punkt\rangle) u$ to denote the function 
$x \mapsto \exp(\bi \kappa \langle (c,x\rangle) u(x)$.
In the special case $n=1$, we simply write $\exp(\bi \kappa c\punkt) u$
for the function $x \mapsto \exp(\bi \kappa c x) u(x)$.

For compact sets $B$, we use the maximum norm
\begin{align*}
  \|f\|_{\infty,B} &:= \max \{ |f(x)|\ :\ x\in B \} &
  &\text{ for all } f\in C(B),
\end{align*}
for bounded linear operators $\Psi \in L(C(B_1),C(B_2))$ that map
functions in $C(B_1)$ to $C(B_2)$, we denote the induced operator norm by
\begin{equation*}
  \|\Psi\|_{\infty,B_2\leftarrow B_1}
  := \sup\left\{ \frac{\|\Psi[f]\|_{\infty,B_2}}{\|f\|_{\infty,B_1}}
         \ :\ f\in C(B_1)\setminus\{0\} \right\}.
\end{equation*}
 
%
%
\section{Directional techniques for oscillatory functions}

\subsection{Polynomials and tensor interpolation}
\label{sec:tensor_interpolation}

Since polynomial interpolation on tensor product domains features prominently 
in the present paper, let us fix some notation and assumptions.
For $m \in \bbbn_0$, we denote by $\Pi_m$ the space of univariate
polynomials of degree $m$.
Let $\xi_0,\ldots,\xi_m \in [-1,1]$  be distinct interpolation 
points and define the associated Lagrange polynomials by 
\begin{align}
\label{eq:lagrange-1D-reference-element}
  L_\nu(z)
  &:= \prod_{\substack{\mu=0\\ \mu\neq\nu}}^m \frac{z-\xi_\mu}{\xi_\nu-\xi_\mu} &
  &\text{ for all } \nu\in[0:m],\ z\in\bbbc. 
\end{align}
The corresponding interpolation operator is given by
\begin{align*}
  \mathfrak{I} \colon C[-1,1] &\to \Pi_m, &
                 f &\mapsto \sum_{\nu=0}^m f(\xi_\nu) L_\nu.
\end{align*}
For a general interval $[a,b]\subseteq\bbbr$, we introduce the affine
mapping
\begin{align*}
  \Phi_{[a,b]} \colon \bbbc &\to \bbbc, &
              z &\mapsto \frac{b+a}{2} + \frac{b-a}{2} z,
\end{align*}
that takes $[-1,1]$ bijectively to $[a,b]$ and define the interpolation
operator
\begin{align*}
  \mathfrak{I}_{[a,b]} \colon C[a,b] &\to \Pi_m, &
                 f &\mapsto \sum_{\nu=0}^m f(\xi_{[a,b],\nu}) L_{[a,b],\nu},
\end{align*}
with the transformed interpolation points and Lagrange polynomials
\begin{align}\label{eq:xi_L_trans_def}
  \xi_{[a,b],\nu} &:= \Phi_{[a,b]}(\xi_\nu), &
  L_{[a,b],\nu} &:= L_\nu \circ \Phi_{[a,b]}^{-1} &
  &\text{ for all } \nu\in[0:m].
\end{align}
Tensor product interpolation on an axis-parallel $n$-dimensional box
\begin{equation*}
  B := [a_1,b_1] \times \cdots \times [a_n,b_n]
\end{equation*}
is defined by combining transformed interpolation points and polynomials
for the $n$ coordinate intervals:
we let
\begin{align*}
  \xi_{B,\nu} &:= (\xi_{[a_1,b_1],\nu_1},\ldots,\xi_{[a_n,b_n],\nu_n}) \in B,\\
  L_{B,\nu}(x) &:= L_{[a_1,b_1],\nu_1}(x_1) \cdots L_{[a_n,b_n],\nu_n}(x_n) &
  &\text{ for all } x\in\bbbc^n,\ \nu\in M:=[0:m]^n
\end{align*}
and define the tensor interpolation operator by
\begin{align}
\label{eq:IB}
  \mathfrak{I}_B[f]
  &= \mathfrak{I}_{[a_1,b_1]}\otimes\cdots\otimes\mathfrak{I}_{[a_n,b_n]}[f]
   = \sum_{\nu\in M} f(\xi_{B,\nu}) L_{B,\nu} &
  &\text{ for all } f\in C(B).
\end{align}

\subsection{Directional multilevel approximation of oscillatory functions}
\label{sec:ideas}

\subsubsection{Directional single level approximation}
\label{sec:single-level}
The factorization (\ref{eq:factorization-of-g}) identifies the function 
\begin{equation*}
  (x,y) \mapsto \exp(\bi \kappa \|x-y\|)
\end{equation*}
as the oscillatory part of $k$.
For large $\kappa$, a separable form cannot be obtained by standard
polynomial approximation, as a large degree would be required. 
In order to overcome this obstacle, we follow an idea of Brandt \cite{BR91}, 
Engquist \& Ying \cite{ENYI07}, Messner \textsl{et al.} \cite{MESCDA12} 
and construct a \emph{directional} approximation: 
we introduce a vector $c\in\bbbr^n$ with $\|c\|=1$ and use
$\langle x-y, c \rangle$ as an approximation of $\|x-y\|$:
\begin{align}
  \exp(\bi \kappa \|x-y\|)
  &= \exp(\bi \kappa \langle x-y, c \rangle)
     \exp(\bi \kappa (\|x-y\| - \langle x-y, c \rangle)\notag\\
  &= \exp(\bi \kappa \langle x, c \rangle)
     \overline{\exp(\bi \kappa \langle y, c \rangle)}
     \exp(\bi \kappa (\|x-y\| - \langle x-y, c \rangle)).
       \label{eq:exp_factorization}
\end{align}
Since the first factor depends only on $x$ and the second only on $y$,
a separable approximation of the third term in (\ref{eq:exp_factorization})
gives rise to a separable approximation of the entire product.
Figure~\ref{fi:cone_smooth} suggests that this term is a smooth function
in a cone extending from zero in the chosen direction $c$, so that standard
polynomial interpolation can be employed.
%
%
\begin{figure}
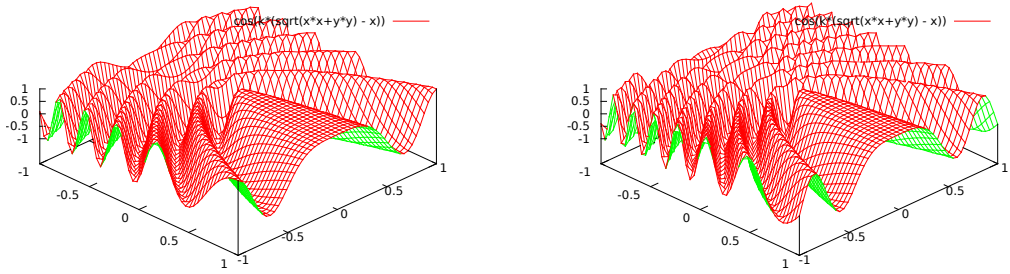

  \pgfdeclareimage[width=7cm]{dirsmooth15c}{fi_dirsmooth15c}
  \pgfdeclareimage[width=7cm]{dirsmooth20c}{fi_dirsmooth20c}
  \begin{center}
    \pgfuseimage{dirsmooth15c}\quad
    \pgfuseimage{dirsmooth20c}
  \end{center}
  \caption{$x\mapsto\cos(\kappa (\|x\|-x_1))$ in $[-1,1]\times[-1,1]$
           for 
           $\kappa \in\{15,20\}$}
  \label{fi:cone_smooth}
\end{figure}
%
Specifically, given axis-parallel target and source boxes
$\tau,\sigma\subseteq\bbbr^n$, we approximate the directionally modified
kernel function
\begin{equation}\label{eq:helmholtz_modified}
  k_c(x,y) := k(x,y) \exp(-\bi \kappa \langle x - y,c\rangle)
\end{equation}
by its interpolating polynomial
\begin{align}
  \tilde k_{c,\tau\sigma}(x,y)
  &:= \sum_{\nu\in M} \sum_{\mu\in M}
        k_c(\xi_{\tau,\nu}, \xi_{\sigma,\mu})
        L_{\tau,\nu}(x) L_{\sigma,\mu}(y) &
  &\text{ for all } x\in \tau,\ y\in \sigma\label{eq:g_c_ts_def}.
\end{align}
Combining the approximation (\ref{eq:g_c_ts_def}) of $k_c$ with
(\ref{eq:exp_factorization}) leads to an approximation of the kernel
function $k$ by 
\begin{align}\label{eq:directional_approximation}
  \tilde k_{\tau\sigma}(x,y)
  &= \sum_{\nu\in M} \sum_{\mu\in M}
     k_c(\xi_{\tau,\nu}, \xi_{\sigma,\mu})
     L_{\tau c,\nu}(x)
     \overline{L_{\sigma c,\mu}(y)} &
  &\text{ for all } x\in \tau,\ y\in \sigma,
\end{align}
where the exponential factors are included in directionally modified Lagrange
polynomials
\begin{align*}
  L_{\tau c,\nu}(x) &:= \exp(\bi \kappa \langle x, c \rangle) L_{\tau,\nu}(x),\\
  L_{\sigma c,\mu}(y) &:= \exp(\bi \kappa \langle y, c \rangle) L_{\sigma,\mu}(y) &
  &\text{ for all } x\in\tau,\ y\in\sigma.
\end{align*}
If $M$ has only few elements, (\ref{eq:directional_approximation}) 
is a short sum of {\em separated} functions.
In fact, if the triple $(\tau,\sigma,c)$ satisfies the
\emph{parabolic admissibility condition} (\ref{eq:admissibility}),
exponential convergence in the polynomial degree $m$ can be expected.
For the three- and two-dimensional Helmholtz kernels, this is proven in
Corollaries~\ref{co:interpolation_error} and
\ref{co:interpolation_error_2D}.
We mention in passing that this result is already sufficient to justify
certain purely algebraic approximation schemes based on orthogonal
factorizations \cite{BO15} that depend only on the existence of
degenerate approximation.

For an analysis of this single-level approximation, it is convenient to 
introduce the plane wave function
\begin{align}
\label{eq:w_c}
  w_c \colon \tau\times\sigma &\to \bbbc, &
           (x,y) &\mapsto \exp(\bi \kappa \langle x-y, c \rangle)
\end{align}
since the approximation (\ref{eq:directional_approximation}) of the
kernel function $k$ can be compactly written as
\begin{equation*}
  \tilde k_{\tau\sigma}
  = w_c\,\mathfrak{I}_{\tau\times\sigma}[\overline{w}_c\,k],
\end{equation*}
i.e., as the combination of multiplication operators and standard
tensor interpolation.

Due to $w_c(x,y)=\exp(\bi \kappa \langle x-y, c \rangle)
=\exp(\bi \kappa \langle x, c \rangle)
 \exp(\bi \kappa \langle y, -c \rangle)$, we have
\begin{equation*}
  \tilde k_{\tau\sigma}
  = \mathfrak{I}_{\tau,c}\otimes\mathfrak{I}_{\sigma,-c}[k]
\end{equation*}
with
\begin{align*}
  \mathfrak{I}_{\tau,c}[u]
  &:= \exp(\bi \kappa \langle c, \punkt \rangle)
      \mathfrak{I}_\tau[\exp(\bi \kappa \langle -c, \punkt \rangle)\,u] &
  &\text{ for all } u\in C(\tau),\\
  \mathfrak{I}_{\sigma,-c}[u]
  &:= \exp(\bi \kappa \langle -c, \punkt \rangle)
      \mathfrak{I}_\sigma[\exp(\bi \kappa \langle c, \punkt \rangle)\,u] &
  &\text{ for all } u\in C(\sigma).
\end{align*}

\subsubsection{Efficient multilevel matrix representation}
\label{sec:algorithm}
The separable structure of the kernel approximation $\tilde k_{\tau\sigma}$
can be exploited algorithmically.
However,  multilevel techniques have additionally to be brought to bear
for the sake of efficiency. 
We describe here a variant that is used in \cite{MESCDA12}. 
To fix ideas, we consider a Galerkin discretization of 
the integral operator $\mathcal{G}$ in the special case that 
$\Gamma \subset \bbbr^n$ is an $(n-1)$-dimensional manifold.
The Galerkin method using a basis $(\varphi_i)_{i\in\Idx}$ leads to the
\emph{stiffness matrix} $G\in\bbbc^{\Idx\times\Idx}$ given by
\begin{align}\label{eq:matrix}
  G_{ij}
  &= \int_\Gamma \varphi_i(x) \int_\Gamma k(x,y) \varphi_j(y) \,dy\,dx &
  &\text{ for all } i,j\in\Idx.
\end{align}
The separable approximation (\ref{eq:directional_approximation}) of the
kernel $k$ can only be used if $k_c$ is sufficiently smooth in
$\tau\times\sigma$, and this is the case if appropriate
\emph{admissibility conditions} hold (below, we will identify 
(\ref{eq:admissibility}) as an appropriate one).  
In order to satisfy the admissibility condition, we recursively split the
index set $\Idx$ into disjoint subsets $\hat\tau$ called \emph{clusters}
and construct axis-parallel \emph{bounding boxes} $\tau$ such that
\begin{align*}
  \supp \varphi_i &\subseteq \tau &
  &\text{ for all } i\in\hat\tau.
\end{align*}
We split $\Idx\times\Idx$ into subsets $\hat\tau\times\hat\sigma$
with clusters $\hat\tau,\hat\sigma$ such that either
\begin{itemize}
  \item $\hat\tau$ and $\hat\sigma$ contain only a small number
    of indices, or
  \item the corresponding boxes $\tau\times\sigma$ satisfy the
    admissibility condition (\ref{eq:admissibility}).
\end{itemize}
In the first case, we store $G|_{\hat\tau\times\hat\sigma}$ directly.
In the second case, we use (\ref{eq:directional_approximation}) to get
\begin{align}
  G_{ij}
  &= \int_\Gamma \varphi_i(x) \int_\Gamma k(x,y) \varphi_j(y) \,dy\,dx
   \approx \int_\Gamma \varphi_i(x) \int_\Gamma
      \tilde k_{\tau\sigma}(x,y) \varphi_j(y) \,dy\,dx\notag\\
  &= \sum_{\nu,\mu\in M}
      \underbrace{k_c(\xi_{\tau,\nu}, \xi_{\sigma,\mu})
                 }_{=: s_{\tau\sigma,\nu\mu}}
      \underbrace{\int_\Gamma \varphi_i(x) L_{\tau c,\nu}(x) \,dx
                 }_{=: v_{\tau c,i \nu}}
      \underbrace{\int_\Gamma \varphi_j(y) \overline{L_{\sigma c,\mu}(y)} \,dy
                 }_{=: \overline{v_{\sigma c,j \mu}}}\notag\\
  &= (V_{\tau c} S_{\tau\sigma} V_{\sigma c}^*)_{ij}
     \qquad\text{ for all } i\in\hat\tau,\ j\in\hat\sigma.
     \label{eq:VSW}
\end{align}
Since the \emph{coupling matrices} $S_{\tau\sigma}\in\bbbc^{M\times M}$
are small, we can afford to store them explicity.

The \emph{basis matrices} $V_{\tau c}\in\bbbc^{\hat\tau\times M}$ are too
large to be stored, but we can take advantage of the fact that the
sets $\hat\tau$ are nodes of a tree so as to obtain
a hierarchical representation of these basis matrices $V_{\tau c}$:
If $\hat\tau$ is a leaf of the tree, we assume that $\hat\tau$ contains
only a few indices and we can afford to store $V_{\tau c}$ directly.
If $\hat\tau$ has a son $\hat\tau'$, we select a direction $c'$ close to
$c$ and use interpolation to approximate
\begin{align}
  L_{\tau c,\nu}(x)
  &= \exp(\bi \kappa \langle x, c \rangle) L_{\tau,\nu}(x)\notag\\
  &= \exp(\bi \kappa \langle x, c' \rangle)
     \exp(\bi \kappa \langle x, c - c' \rangle)
     L_{\tau,\nu}(x)\notag\\
  &\approx \exp(\bi \kappa \langle x, c' \rangle)
     \sum_{\nu'\in M}
     \underbrace{
       \exp(\bi \kappa \langle \xi_{\tau',\nu'}, c - c' \rangle)
       L_{\tau,\nu}(\xi_{\tau',\nu'})
     }_{=: e_{\tau'\tau c,\nu'\nu}} 
     L_{\tau',\nu'}(x).\label{eq:lagrange_reinterpolation}
\end{align}
Therefore
\begin{equation}\label{eq:nested}
  V_{\tau c}|_{\hat\tau'\times M}
  \approx V_{\tau' c'} E_{\tau'\tau c}
\end{equation}
with the \emph{transfer matrix} $E_{\tau'\tau c}\in\bbbc^{M\times M}$.
If we replace $V_{\tau c}$ by an approximation $\widetilde{V}_{\tau c}$
given by $\widetilde{V}_{\tau c}=V_{\tau c}$ if $\hat\tau$ is a leaf and
\begin{align}\label{eq:V_def}
  \widetilde{V}_{\tau c}|_{\hat\tau'\times M}
  &:= \widetilde{V}_{\tau' c} E_{\tau'\tau c} &
  &\text{ for all sons } \hat\tau' \text{ of } \hat\tau
\end{align}
otherwise, we obtain the \emph{directional $\mathcal{H}^2$-matrix
approximation} of $G$.

%
%
\begin{remark}[Directions]
\label{re:directions}
In an algorithmic realization of directional $\mathcal{H}^2$-matrices,
a finite set of directions is associated with each cluster $\hat\tau$.
The necessary directional resolution is given by the admissibility condition
(\ref{eq:adm_directional}), and there are simple algorithms
\cite{BO15,BEKUVE15} for constructing suitable directions.

In spite of the fact that large clusters require a large number of
directions, it can be shown that in typical situations the resulting
directional $\mathcal{H}^2$-matrices require
$\mathcal{O}(N k + \kappa^2 k^2 \log N)$ units of storage,
where $N=\#\Idx$, $k=\#M = (m+1)^n$ \cite{BO15,MESCDA12}.
If $\mathcal{H}$-matrices are used to treat the low-frequency
case, we obtain a similar complexity estimate \cite{BEKUVE15}.
\end{remark}

%
%
\begin{remark}[Algebraic recompression]
In the interest of efficiency, the algorithm in \cite{MESCDA12} 
combines the above techniques with further algebraic recompression 
of the coupling matrices. 
\end{remark}

%
%
\begin{remark}[Nested multilevel]
Nested multilevel structures are essential for poly\-lo\-ga\-rithmic-linear
complexity in the high-frequency setting.
Instead of resorting to interpolation to set up the factorizations
(\ref{eq:VSW}) and (\ref{eq:V_def}), it is also possible to construct
them directly via approximate rank-revealing factorizations based on
a heuristic pivoting strategy as proposed in \cite{BEKUVE15}. 
\end{remark}

\subsubsection{Error analysis via multilevel approximation of the kernel}
\label{sec:error-analysis-via-multilevel}
The goal of the article is to provide a rigorous error analysis 
for the various approximation steps in Section~\ref{sec:algorithm}. 
In the present section, we set the stage for the error analysis by 
casting the analysis in the framework of polynomial interpolation.

Let $(\tau,\sigma,c)$ satisfy the admissibility condition
(\ref{eq:admissibility}).
We fix sequences
\begin{align*}
  \tau_0 &\supseteq \tau_1 \supseteq \cdots \supseteq \tau_L, &
  \sigma_0 &\supseteq \sigma_1 \supseteq \cdots \supseteq \sigma_L
\end{align*}
of axis-parallel boxes and a sequence $c_0,\ldots,c_L\in\bbbr^n$ of
directions such that
\begin{itemize}
  \item $\tau_0=\tau$, $\sigma_0=\sigma$, $c_0=c$,
  \item $\tau_\ell$ is a son of $\tau_{\ell-1}$,
        $\sigma_\ell$ is a son of $\sigma_{\ell-1}$ for all $\ell\in[1:L]$,
        and
  \item $\tau_L$ and $\sigma_L$ are leaf clusters.
\end{itemize}
The re-interpolation (\ref{eq:lagrange_reinterpolation}) means that
we replace
\begin{align}
  v_{\tau c,i\nu}
  &= \int_\Gamma \varphi_i(x) L_{\tau c,\nu}(x) \,dx &
  &\text{ by the approximation}\notag\\
  \widetilde{v}_{\tau c,i\nu}
  &= \int_\Gamma \varphi_i(x)
     \underbrace{\mathfrak{I}_{\tau_L,c_L} \circ \cdots
                 \circ \mathfrak{I}_{\tau_1,c_1}[L_{\tau c,\nu}](x)
                }_{=:\widetilde{L}_{\tau c,\nu}(x)} \,dx &
  &\text{ for all } i\in\hat\tau_L,\ \nu\in M.\label{eq:Ltilde}
\end{align}
Here we use the notation $\widetilde{L}_{\tau c,\nu}$ for the sake of brevity,
in spite of the fact that it depends on the entire sequences
$\tau_0,\ldots,\tau_L$ and $c_0,\ldots,c_L$.
In particular, different approximations are used for different leaf
clusters $\tau_L$.
We can analyze the re-interpolation error by gauging the difference
between $L_{\tau c,\nu}$ and $\widetilde{L}_{\tau c,\nu}$.

An alternative, very closely related approach to estimating the accuracy
of the algorithm described in Section~\ref{sec:algorithm} is to study the
effect of interpolating the kernel functions $k$:
we denote the corresponding interpolation operators by
\begin{align}\label{eq:directional-interpolation-operators}
  \mathfrak{I}_{\tau_\ell \times \sigma_\ell,c_\ell}
  &:= \mathfrak{I}_{\tau_\ell,c_\ell}
      \otimes \mathfrak{I}_{\sigma_\ell,-c_\ell} &
  &\text{ for all } \ell\in[0:L],
\end{align}
and note that our nested interpolation scheme approximates
$k|_{\tau_L\times\sigma_L}$ by
\begin{equation}\label{eq:nested_interpolation}
  \hat k_{\tau\sigma}
  := \mathfrak{I}_{\tau_L \times \sigma_L,c_L} \circ \cdots
      \circ \mathfrak{I}_{\tau_0 \times \sigma_0,c_0}[k].
\end{equation}
The analysis of the algorithm described in Section~\ref{sec:algorithm} 
amounts to estimating the error $k - \hat k_{\tau \sigma}$.
Writing the error as a telescoping sum
\begin{equation*}
  k - \hat k_{\tau\sigma}
  = \sum_{\ell=0}^{L-1} \mathfrak{I}_{\tau_L \times \sigma_L,c_L} \circ \cdots
     \circ \mathfrak{I}_{\tau_{\ell+1} \times \sigma_{\ell+1}}
                  [k - \mathfrak{I}_{\tau_\ell\times\sigma_\ell,c_\ell}[k]]
\end{equation*}
splits the error analysis into two parts: the analysis of
the interpolation errors
$k-\mathfrak{I}_{\tau_\ell \times \sigma_\ell,c_\ell}[k]$ that is the topic
of Sections~\ref{sec:g3D} and \ref{sec:other-kernels}, 
and a stability analysis of the nested interpolation operators
$\mathfrak{I}_{\tau_L \times \sigma_L,c_L}\circ\cdots\circ
\mathfrak{I}_{\tau_{\ell+1} \times \sigma_{\ell+1},c_{\ell+1}}$ for all
$\ell\in[0:L-1]$.
The latter stability analysis is the topic of
Section~\ref{sec:nested_approximation}, where we work out 
how the shrinking rate of the nested boxes
$\tau_\ell \times \sigma_\ell$, $\ell \in [0:L]$
and the differences $\|c_\ell - c_{\ell+1}\|$ of two consecutive directions 
impact the stability of the iterated operator 
$\mathfrak{I}_{\tau_L \times \sigma_L,c_L}\circ\cdots\circ
\mathfrak{I}_{\tau_{\ell+1} \times \sigma_{\ell+1},c_{\ell+1}}$.  
In the interest of future reference, we formulate our findings as 
\begin{align}
\label{eq:nested-interpolation-error-vorne}
  &\|k - \hat k_{\tau\sigma}\|_{\infty,\tau_L \times \sigma_L} 
  \leq \\ 
  & \quad \sum_{\ell=0}^{L-1}
      \| \mathfrak{I}_{\tau_L \times \sigma_L,c_L} \circ \cdots \circ 
         \mathfrak{I}_{\tau_{\ell+1}\times \sigma_{\ell+1},c_{\ell+1}}
      \|_{\infty,\tau_L\times\sigma_L\leftarrow\tau_{\ell+1}\times\sigma_{\ell+1}}
      \|k - \mathfrak{I}_{\tau_\ell\times\sigma_\ell,c_\ell}[k]
      \|_{\infty,\tau_{\ell+1} \times \sigma_{\ell+1}}.\notag
\end{align}

%
%
\section{Single level analysis for the three-dimensional case}
\label{sec:g3D}

We focus here on the analysis of the 3D Helmholtz kernel, i.e.,
we consider $k=g$ (with $g_c$, $\tilde g_{\tau\sigma}$ and
$\tilde g_{c,\tau\sigma}$ defined accordingly).
The 2D Helmholtz kernel $h$ can be studied using similar techniques and 
is discussed briefly in Section~\ref{sec:other-kernels}. 
We mention that our single-level analysis differs from \cite{BEKUVE15,ENYI07} 
in the technique employed; that is, we opted for a ``derivative-free'' approach 
based on complex analysis.

For bounding boxes $\tau,\sigma\subseteq\bbbr^n$ and a direction $c\in\bbbr^n$,
we immediately find
\begin{align*}
  |\exp(\bi \kappa \langle x-y, c \rangle)|
  &= 1 &
  &\text{ for all } x\in \tau,\ y\in \sigma,
\end{align*}
and we can conclude that multiplication with a plane wave does not
change the maximum norm.
This implies
\begin{equation}\label{eq:exp_interpolation_error}
  \| g - \tilde g_{\tau\sigma} \|_{\infty,\tau\times \sigma}
  = \| g_c - \tilde g_{c,\tau\sigma} \|_{\infty,\tau\times \sigma}
\end{equation}
for the approximations $\tilde g_{\tau\sigma}$ and $\tilde g_{c,\tau\sigma}$
defined in (\ref{eq:directional_approximation}) and (\ref{eq:g_c_ts_def}).
This equation allows us to focus on interpolation error estimates
for the directionally modified function $g_c$.

\subsection{Tensor interpolation}
\label{sub:tensor_interpolation}

The error analysis of our scheme has to gauge two sources
of error: the interpolation error associated with 
(\ref{eq:g_c_ts_def}) and the interpolation error arising 
from the nested interpolation  (\ref{eq:nested_interpolation}).
Both cases require error estimates for tensor interpolation. 

Throughout the article, the \emph{Lebesgue constant} $\Lambda_m$ is given by 
\begin{equation}\label{eq:Lambda_m} 
  \Lambda_m := \|\mathfrak{I}\|_{\infty,[-1,1]\leftarrow[-1,1]},
\end{equation}
and we will make the following assumption on $\mathfrak{I}$: 

%
%
\begin{assumption}
\label{assumption:Lambda}
There are constants $C_\Lambda \in \bbbr_{> 0}$ and
$\lambda \in \bbbr_{\geq 1}$ such that
\begin{align}\label{eq:polynomial-growth-Lambda}
  \Lambda_m &\leq C_\Lambda (m+1)^\lambda &
  &\text{ for all } m\in\bbbn_0.
\end{align}
\end{assumption}

%
%
\begin{remark}[Chebyshev interpolation]
A good choice for ${\mathfrak I}$ is interpolation in the 
Chebyshev points $\xi_\nu = \cos\left(\frac{2\nu+1}{2m+2}\right)$,
$\nu\in[0:m]$.
In this case, \cite{RI90} gives 
$\Lambda_m \leq \frac{2}{\pi} \ln (m+1)+1$. Thus, 
Chebyshev interpolation satisfies Assumption~\ref{assumption:Lambda}
with $C_\Lambda=\lambda=1$.
\end{remark}

We can apply tensor arguments to extend the 1D stability assumptions to
the multi-dimensional case.
Let us consider an axis-parallel box
\begin{equation}\label{eq:box-B}
  B = [a_1,b_1]\times\cdots\times[a_n,b_n].
\end{equation}
Recall the tensor product interpolation operator ${\mathfrak I}_{B}$ from 
(\ref{eq:IB}) that is obtained from the 1D interpolation operator
${\mathfrak I}$.
This operator can be written as a product of partial interpolation
operators $I\otimes\cdots\otimes
\mathfrak{I}_{[a_\iota,b_\iota]}\otimes\cdots\otimes I$ that each apply
interpolation only in one coordinate direction $\iota\in[1:n]$.
Since the stability estimate (\ref{eq:Lambda_m}) carries over to
these operators, a simple telescoping sum and the relationship between
partial interpolation and one-dimensional interpolation can be used
to prove the following estimate:

%
%
\begin{lemma}[Tensor interpolation]
\label{le:tensor_interpolation}
Let $B$ be given by (\ref{eq:box-B}). 
For $f\in C(B)$ define
\begin{align}\label{eq:f_x_iota_def}
  f_{x,\iota} \colon [-1,1] &\to \bbbr, &
       t &\mapsto f\left(x_1,\ldots,x_{\iota-1},
                         \Phi_{[a_\iota,b_\iota]}(t),
                         x_{\iota+1},\ldots,x_n\right)
\end{align}
for all $x\in B$ and $\iota\in[1:n]$.
Then 
\begin{equation*}
  \|f-\mathfrak{I}_B[f]\|_{\infty,B}
  \leq \sum_{\iota=1}^n \Lambda_m^{\iota-1}
      \max\{\|f_{x,\iota} - \mathfrak{I}[f_{x,\iota}]\|_{\infty,[a_\iota,b_\iota]}
                 \ :\ x\in B \}.
\end{equation*}
\end{lemma}

In the setting of Lemma~\ref{le:tensor_interpolation} we can find $d$, $p\in\bbbr^n$ such that
\begin{align*}
  d-tp &\in B, &
  f_{x,\iota}(t) &= f(d-tp) &
  &\text{ for all } t\in[-1,1].
\end{align*}
Hence, it suffices to bound the interpolation errors for all functions of the form 
\begin{align*}
  f_{dp} \colon [-1,1] &\to \bbbc, &
      t &\mapsto f(d-tp),
\end{align*}
with $d$, $p\in\bbbr^n$ such that $d-tp\in B$ for all $t\in[-1,1]$.
Note that the latter condition implies $2 \|p\|\leq\diam(B)$.

We are interested in interpolating the function $g_c$.
Taking advantage of the fact that it only depends on the relative
coordinates $x-y$, we obtain the following result:

%
%
\begin{lemma}[Univariate formulation for $g$]
\label{le:univariate_formulation}
Let $\epsilon\in\bbbr_{\geq 0}$, and let $\tau,\sigma\subseteq\bbbr^3$
be axis-parallel boxes.
If we have
\begin{equation}\label{eq:g1_interpolation}
  \|g_{dp} - \mathfrak{I}[g_{dp}]\|_{\infty,[-1,1]} \leq \epsilon
\end{equation}
with
\begin{align}\label{eq:g1_def}
  g_{dp} \colon [-1,1] &\to \bbbc, &
        t &\mapsto \frac{\exp(\bi \kappa (\|d-tp\| - \langle d-tp,c \rangle))}
                        {4 \pi \|d-tp\|},
\end{align}
for all $d$, $p\in\bbbr^3$ that satisfy
\begin{subequations}
\begin{align}
  2 \|p\|
  &\leq \max\{\diam(\tau),\diam(\sigma)\}\label{eq:p_bound},\\
  d-t p &\in \tau-\sigma = \{ x-y\ :\ x\in \tau,\ y\in \sigma \} &
  &\text{ for all } t\in[-1,1],\label{eq:dtp_bound}
\end{align}
\end{subequations}
the directional approximation (\ref{eq:directional_approximation})
error is bounded by
\begin{equation}
  \|g - \tilde g_{\tau\sigma}\|_{\infty,\tau\times \sigma} \leq 6 \Lambda_m^5 \epsilon.
\end{equation}
\end{lemma}
\begin{proof}
We apply Lemma~\ref{le:tensor_interpolation} to
$f:=g_c\in C(\tau\times\sigma)$.
Given $x\in\tau$ and $y\in\sigma$, the functions $f_{(x,y),\iota}$,
$\iota\in[1:6]$, coincide with $g_{dp}$ for certain vectors $d,p\in\bbbr^3$
satisfying (\ref{eq:p_bound}), (\ref{eq:dtp_bound}).
Since we have (\ref{eq:g1_interpolation}) at our disposal for \emph{all}
of these pairs of vectors, we obtain the required estimate for
$\|g_c-\tilde g_{c,\tau\sigma}\|_{\infty,\tau\times\sigma}
= \|g-\tilde g_{\tau\sigma}\|_{\infty,\tau\times\sigma}$.
\end{proof}

\subsection{\texorpdfstring{Holomorphic extension of $g_{dp}$}
                           {Holomorphic extension of gdp}}

In order to obtain bounds for the interpolation error of the
functions $g_{dp}$ defined in (\ref{eq:g1_def}), we consider its
holomorphic extension into a neighborhood of the interval $[-1,1]$. 
The key step is to understand the extension of $t \mapsto \|d - t p\|$. 
In turn, the holomorphic extension of the Euclidean norm
\begin{align*}
  \|x\| &= \sqrt{\langle x, x \rangle} &
  &\text{ for all } x\in\bbbr^n
\end{align*}
requires a suitable extension of the square root, which cannot be 
defined in all of $\bbbc$. We choose the principal branch given by
\begin{align*}
  \sqrt{z} &= \sqrt{|z|} \frac{z+|z|}{|z+|z||} &
  &\text{ for all } z\in\bbbc\setminus\bbbr_{\leq 0},
\end{align*}
which is holomorphic in $\bbbc\setminus\bbbr_{\leq 0}$ and maps to
$\bbbc_+ := \{ z\in\bbbc\ :\ \Re(z)>0 \}$.
In order to identify a subset of $\bbbc$ in which 
$z\mapsto \sqrt{\langle d-zp, d-zp \rangle}$ is holomorphic,
we have to determine the values $z\in\bbbc$ satisfying
\begin{equation*}
  \langle d-zp, d-zp \rangle \not\in\bbbr_{\leq 0}.
\end{equation*}

%
%
\begin{lemma}[Extension of the Euclidean norm]
\label{le:holomorphic_extension}
Let $n\in\bbbn$ and $d,p\in\bbbr^n$ with $p\neq 0$ and define
\begin{gather*}
  w_r := \langle d, p \rangle / \|p\|^2,\qquad
  w_i := \sqrt{\|d\|^2/\|p\|^2 - w_r^2},\qquad
  w := w_r + \bi w_i,\\
  U_{dp} := \bbbc \setminus \{ w_r+\bi y\ :\ y\in\bbbr,\ |y|\geq w_i \}.
\end{gather*}
We have
\begin{align}\label{eq:product_pwz}
  \langle d-zp,d-zp \rangle &= \|p\|^2 (w-z) (\bar w-z) &
  &\text{ for all } z\in\bbbc.
\end{align}
The function
\begin{align}
\label{eq:enorm}
  \enorm \colon U_{dp} &\to \bbbc_+, &
     z &\mapsto \sqrt{\langle d-zp, d-zp \rangle}
                = \|p\| \sqrt{(w-z) (\bar w-z)},
\end{align}
is well-defined and holomorphic.
\end{lemma}
\begin{proof}
The equality (\ref{eq:product_pwz}) follows from a direct computation using 
$d,p\in\bbbr^n$ and $|w|=\|d\|/\|p\|$:
\begin{align*}
  \langle d - z p, d - zp\rangle
  &= \|d\|^2 - 2 \langle d,p\rangle z + \|p\|^2 z^2
   = \|p\|^2 |w|^2 - \|p\|^2 (w + \bar w) z + \|p\|^2 z^2\\
  &= \|p\|^2 (w - z) (\bar w - z).
\end{align*}
In order to show that $\enorm$ is well-defined, it suffices to demonstrate 
\begin{align*}
  z \in U_{dp} &\Longrightarrow
  \langle d-zp, d-zp \rangle \in \bbbc\setminus\bbbr_{\leq 0} &
  &\text{ for all } z\in\bbbc.
\end{align*}
We use contraposition: We let $z\in\bbbc$ be such that
$\langle d-zp, d-zp \rangle\in\bbbr_{\leq 0}$ and prove that this
implies $z\not\in U_{dp}$.
Let $x,y\in\bbbr$ with $z=x+\bi y$.
We have
\begin{align*}
  \langle d-zp, d-zp \rangle
  &= \|p\|^2 (w-z) (\bar w-z)\\
  &= \|p\|^2 ((w_r - x) + \bi(w_i - y))\,
             ((w_r - x) + \bi(-w_i - y))\\
  &= \|p\|^2 ((w_r - x)^2 - 2 \bi (w_r - x) y + w_i^2 - y^2).
\end{align*}
Due to $\langle d-zp, d-zp \rangle\in\bbbr_{\leq 0}$, the imaginary
part vanishes and the real part is non-positive, so $\|p\|>0$ yields
\begin{align*}
  0 &= 2 (w_r - x) y, &
  0 &\geq (w_r - x)^2 + w_i^2 - y^2.
\end{align*}
If $y=0$, then the inequality gives us $x=w_r$ and $w_i=0\leq |y|$,
i.e., $z\not\in U_{dp}$.
Otherwise, the equation yields $x=w_r$ and the inequality $y^2\geq w_i^2$,
i.e., again $z\not\in U_{dp}$.

Since $z\mapsto \langle d-zp, d-zp \rangle$ is holomorphic in $U_{dp}$ and
maps into the domain of the holomorphic principal square root, the
composed function $\enorm$ is also holomorphic.
\end{proof}

%
%
\begin{figure}
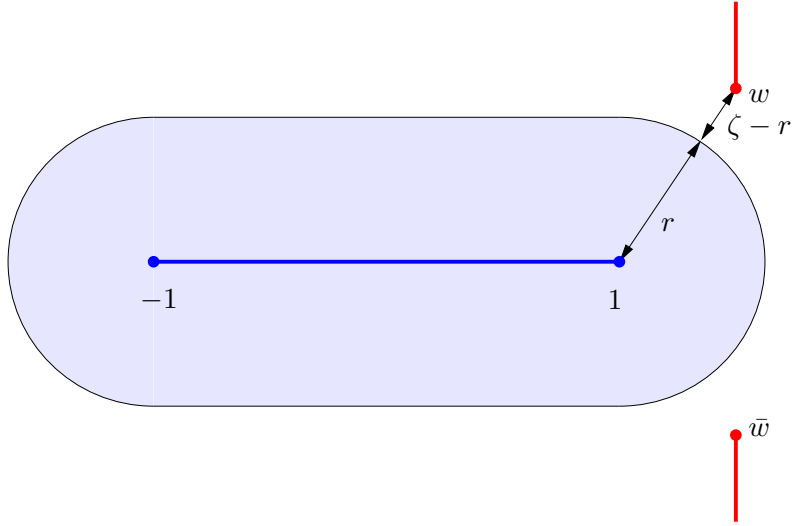

  \pgfdeclareimage[width=10cm]{UrUd}{fi_UrUd}
  \begin{center}
    \pgfuseimage{UrUd}%
    \pgfputat{\pgfxy(-8,3)}{\pgfbox[center,center]{$-1$}}%
    \pgfputat{\pgfxy(-2,3)}{\pgfbox[center,center]{$1$}}%
    \pgfputat{\pgfxy(-0.1,5.7)}{\pgfbox[center,center]{$w$}}%
    \pgfputat{\pgfxy(-0.1,1.3)}{\pgfbox[center,center]{$\bar w$}}%
    \pgfputat{\pgfxy(-1.3,4)}{\pgfbox[center,center]{$r$}}%
    \pgfputat{\pgfxy(-0.1,5.3)}{\pgfbox[center,center]{$\zeta-r$}}%
  \end{center}
  \caption{Domain $U_r$ in relation to the interval $[-1,1]$ and $\{w,\bar w\}$}
\end{figure}

The function $\enorm$ (and thus also $g_{dp}$) is holomorphic on $U_{dp}$. The singularities 
of $\enorm$ closest to the interval $[-1,1]$ are the branch points $w$ and $\bar w$. 
To compute their distance from $[-1,1]$ we use (\ref{eq:product_pwz}) to find for $t \in [-1,1]$ 
\begin{equation}\label{eq:wt_dtp}
  |w-t|^2
  = (w-t) (\bar w-t)
  = \frac{\langle d-tp, d-tp \rangle}{\|p\|^2}
  = \frac{\|d-tp\|^2}{\|p\|^2}
\end{equation}
so that the distance is given by
\begin{equation}\label{eq:zeta_def}
  \zeta := \min\{ |w-t|\ :\ t\in[-1,1] \}
  = \min\left\{ \frac{\|d-tp\|}{\|p\|}\ :\ t\in[-1,1] \right\}.
\end{equation}

%
%
\begin{lemma}[Bound for $\enorm$]
\label{le:norm_estimates}
Let $n\in\bbbn$ and $d$, $p\in\bbbr^n$ with $p\neq 0$,
and let $\enorm$, $w$, and $U_{dp}$ be defined as in
Lemma~\ref{le:holomorphic_extension}.
Let $r\in[0,\zeta)$ and
\begin{equation}\label{eq:Ur_def}
  U_r := \{ z\in\bbbc\ :\ \exists t\in[-1,1]\ :\ |z-t|\leq r \}.
\end{equation}
Then, we have $U_r\subseteq U_{dp}$ and
\begin{align}
\label{eq:lemma:norm_estimates-10}
  |\enorm(z)| &\geq \|p\|(\zeta-r) &
  &\text{ for all } z\in U_r.
\end{align}
\end{lemma}
\begin{proof}
We prove $U_r\subseteq U_{dp}$ by contraposition.
Let $z\in\bbbc\setminus U_{dp}$.
This implies $z=w_r+\bi y$ with $y\in\bbbr$ and $y^2\geq w_i^2$.
Due to (\ref{eq:wt_dtp}), we have
\begin{equation*}
  |z-t|^2 = (w_r-t)^2 + y^2 \geq (w_r-t)^2 + w_i^2 = |w-t|^2
  = \frac{\|d-tp\|^2}{\|p\|^2} \geq \zeta^2 > r^2
\end{equation*}
and therefore $z\not\in U_r$.

Having proven $U_r\subseteq U_{dp}$, we show 
the lower bound (\ref{eq:lemma:norm_estimates-10}). 
Let $z\in U_r$.
We can find $t\in[-1,1]$ such that $|z-t|\leq r$.
Using (\ref{eq:zeta_def}), this implies
\begin{align*}
  |w-z| &= |w-t+t-z| \geq |w-t| - |t-z| \geq \zeta - r > 0,\\
  |\bar w-z| &= |\bar w-t+t-z| \geq |\bar w-t| - |t-z|
         = |w-t| - |t-z| \geq \zeta - r > 0.
\end{align*}
The proof is completed by observing
\begin{equation*}
  |\enorm(z)|
  = \|p\| \sqrt{|w-z|\,|\bar w-z|}
  \geq \|p\| (\zeta-r).
\end{equation*}
\end{proof}

We also require a bound for
$z \mapsto \exp(\bi \kappa (\enorm(z)-\langle d - zp,c\rangle))$.
Noting 
\begin{align*}
  |\exp(\bi \kappa  (\enorm(z) - \langle d-zp, c \rangle))|
  &\leq \exp(\kappa  |\Im(\enorm(z) - \langle d-zp,c \rangle)|) &
  &\text{ for all } z\in U_{dp},
\end{align*}
our next goal is to find an upper bound for
$|\Im(\enorm(z) - \langle d-zp,c \rangle)|$.
Following the approach of \cite{MESCDA12}, we apply a Taylor expansion
of $\enorm$ around $t\in[-1,1]$ to estimate this in Lemma~\ref{le:taylor} below.

%
%
\begin{lemma}
\label{le:integral}
Let $\zeta_t\in\bbbr_{>0}$ and $r\in[0,\zeta_t)$.
We have
\begin{equation*}
  \int_0^1 \frac{1-s}{(\zeta_t-rs)^3} \,ds = \frac{1}{2 \zeta_t^2 (\zeta_t-r)}.
\end{equation*}
\end{lemma}
\begin{proof}
The proof is straightforward for $r=0$.
For $r>0$, we note that the function
\begin{align*}
  [0,1] &\to \bbbr, &
  s &\mapsto \frac{(1+\zeta_t/r) - 2 s}
                  {2 r (\zeta_t - rs)^2},
\end{align*}
is the antiderivative of the integrand. 
\end{proof}

%
%
\begin{lemma}[Exponent bound]
\label{le:taylor}
Let $d$, $p\in\bbbr^3$ with $p\neq 0$ and $c\in\bbbr^3$ with $\|c\|=1$.
Let $r\in[0,\zeta)$ with $\zeta$ given in (\ref{eq:zeta_def}), 
and let $U_r$, $U_{dp}$ be defined as
in Lemmas~\ref{le:holomorphic_extension}, \ref{le:norm_estimates}. 
For every $z\in U_r$, there is a $t\in[-1,1]$ such that
\begin{equation}\label{eq:fe_bound}
  |\Im(\enorm(z) - \langle d-zp,c \rangle)|
  \leq \|p\| \left( \left\|\frac{d-tp}{\|d-tp\|} - c \right\| r
                     + \frac{1}{2 (\zeta-r)} r^2 \right).
\end{equation}
\end{lemma}
\begin{proof}
Let $z\in U_r$.
Due to Lemma~\ref{le:norm_estimates}, this implies $z\in U_{dp}$.
By definition, we can find $t\in[-1,1]$ with $|z-t|\leq r$.
We have
\begin{align*}
  \enorm(z) &= \sqrt{\langle d-zp, d-zp \rangle}
   = \langle d-zp, d-zp \rangle^{1/2},\\
  \enorm'(z) &= - 
\frac{\langle p, d-zp \rangle}{\langle d-zp, d-zp \rangle^{1/2}}
   = \frac{-\langle p, d-zp \rangle}{\enorm(z)},\\
  \enorm''(z) &= \frac{\langle p, p \rangle \enorm(z)
                  + \langle p, d-zp \rangle \enorm'(z)}
                 {\enorm(z)^2}
   = \frac{\langle p, p \rangle \enorm(z)
           - \langle p, d-zp \rangle^2 / \enorm(z)}
          {\enorm(z)^2}\\
  &= \frac{\langle p, p \rangle \enorm(z)^2
           - \langle p, d-zp \rangle^2}
          {\enorm(z)^3}
   = \frac{\langle d-zp, d-zp \rangle \langle p, p \rangle
           - \langle p, d-zp \rangle^2}
          {\enorm(z)^3}.
\end{align*}
We use a Taylor expansion of $\enorm$ around $t$. More precisely, 
with the parametrization
\begin{align*}
  \hat z \colon [0,1] &\to U_{dp}, &
                s &\mapsto \hat z_s := t + (z-t) s,
\end{align*}
we have $\hat z_0=\hat z(0)=t$, $\hat z_1=\hat z(1)=z$, and
$\hat z'(s)=z-t$ for all $s\in[0,1]$,
and the Taylor expansion of $\enorm \circ \hat z$ around $s=0$  yields
\begin{equation*}
  \enorm(z) = (\enorm\circ\hat z)(1)
    = \enorm(t) + \enorm'(t) (z-t)
       + \int_0^1 \enorm''(\hat z_s) (1-s) \,ds (z-t)^2.
\end{equation*}
Hence, we obtain the equation
\begin{align*}
  \enorm(z) &- \langle d-zp, c \rangle\\
  &= \enorm(t) + \enorm'(t) (z-t)
        + \int_0^1 \enorm''(\hat z_s) (1-s) \,ds (z-t)^2
   - \langle d-zp, c \rangle\\
  &= \|d-tp\| - \langle d-zp, c \rangle
        - \frac{\langle d-tp, p \rangle}{\|d-tp\|} (z-t)\\
  &\qquad + \int_0^1 \frac{\langle d-\hat z_s p, d-\hat z_s p \rangle
                        \langle p, p \rangle
                        - \langle p, d-\hat z_s p \rangle^2}
                        {\enorm(\hat z_s)^3} (1-s) \,ds (z-t)^2\\
  &= \left\langle d-zp, \frac{d-tp}{\|d-tp\|} - c \right\rangle\\
  &\qquad + \int_0^1 \frac{\langle d-\hat z_s p, d-\hat z_s p \rangle
                        \langle p, p \rangle
                        - \langle p, d-\hat z_s p \rangle^2}
                        {\enorm(\hat z_s)^3} (1-s) \,ds (z-t)^2 \\
&=: S_1 + S_2. 
\end{align*}
We take a closer look at the integrand of $S_2$. 
For any $z\in\bbbc$, we find
\begin{equation}\label{eq:cauchy_schwarz_z}
  \langle d-zp, d-zp \rangle \langle p, p \rangle
     - \langle d-zp, p \rangle^2
  = \|d\|^2 \|p\|^2 - \langle d, p \rangle^2.
\end{equation}
For any $s\in[0,1]$, applying (\ref{eq:cauchy_schwarz_z}) twice produces 
\begin{align}
\nonumber 
  |\langle d-\hat z_s p &, d-\hat z_s p \rangle \langle p, p \rangle
      - \langle d-\hat z_s p, p \rangle^2|
   = \left| \|d\|^2 \|p\|^2 - \langle d, p \rangle^2 \right| 
   =  \|d\|^2 \|p\|^2 - \langle d, p \rangle^2 
\\
  &= \langle d-t p, d-t p \rangle \langle p, p \rangle
        - \langle d-t p, p \rangle^2
   \leq \|d - t p\|^2 \|p\|^2.
\label{eq:foo}
\end{align}
Using (\ref{eq:product_pwz}), we obtain 
\begin{equation}
\label{eq:fooo}
  \frac{\|d-tp\|^2 \|p\|^2}{|\enorm(\hat z_s)^3|}
   = \frac{\|p\|^4 |w-t|^2}{\|p\|^3 |w-\hat z_s|^3}
   \leq \frac{\|p\|\,|w-t|^2}{(|w-t|-|z-t|s)^3}
   \leq \frac{\|p\|\,|w-t|^2}{(|w-t|-rs)^3}.
\end{equation}
Inserting (\ref{eq:foo}) and (\ref{eq:fooo}) in the integrand of $S_2$ 
and applying Lemma~\ref{le:integral} with $\zeta_t := |w-t|$ yields
\begin{align*}
| S_2|    &\stackrel{(\ref{eq:foo})}{\leq} \left|\int_0^1 \frac{\|d-tp\|^2 \|p\|^2 }
                       {|\enorm(\hat z_s)|^3}
           (1-s) \,ds (z-t)^2 \right|\\
  &\qquad \stackrel{(\ref{eq:fooo})}{\leq} \|p\|\,|w-t|^2
                \int_0^1 \frac{1-s}{(\zeta_t-rs)^3} \,ds |z-t|^2
   = \frac{\|p\|\,|w-t|^2}{2 \zeta_t^2 (\zeta_t-r)} |z-t|^2
     \leq \frac{\|p\|}{2 (\zeta_t-r)} r^2.
\end{align*}
In order to estimate $S_1$, we write $z = x + \bi y$ with $x$, $y\in\bbbr$.  
This implies $z-t=(x-t)+\bi y$ and $|z-t|\geq |y|$. With 
the Cauchy-Schwarz inequality we get  
\begin{align*}
|\Im S_1 | & = 
  \left|\Im\left(\left\langle d-zp, \frac{d-tp}{\|d-tp\|}
         - c \right\rangle\right)\right|
  = \left|\left\langle yp, \frac{d-tp}{\|d-tp\|} - c \right\rangle\right|
   \leq |y|\,\|p\| \left\| \frac{d-tp}{\|d-tp\|} - c \right\|\\
  &\leq \|p\| \left\| \frac{d-tp}{\|d-tp\|} - c \right\| |z-t|
   \leq \|p\| \left\| \frac{d-tp}{\|d-tp\|} - c \right\| r.
\end{align*}
Combining the estimates for $|\Im S_1|$ and $|S_2$ with $\zeta_t\geq\zeta$
gives us (\ref{eq:fe_bound}).
\end{proof}

\subsection{Admissibility conditions}
\label{su:admissibility}

In order to obtain useful estimates for the interpolation error of $g_{dp}$,
we have to control the absolute value of the holomorphic
extension
\begin{align*}
  g_{dp} \colon U_r &\to \bbbc, &
       z &\mapsto \frac{\exp(\bi \kappa (\enorm(z)
                                         - \langle d-zp,c \rangle))}
                       {4 \pi \enorm(z)},
\end{align*}
of $g_{dp}$ (cf. (\ref{eq:g1_def})). 
For the denominator, $d-tp\in \tau-\sigma$ implies that we have to ensure
that $\tau$ and $\sigma$ are well-separated.
This is guaranteed by the \emph{standard admissibility condition}
\begin{subequations}
\label{eq:admissibility}
\begin{equation}\label{eq:adm_standard}
  \max\{\diam(\tau),\diam(\sigma)\} \leq \eta_2 \dist(\tau,\sigma),
\end{equation}
where $\eta_2\in\bbbr_{>0}$ is a parameter that can be chosen to balance
the computational complexity and the speed of convergence.

For the numerator, we have
\begin{equation*}
  |\exp(\bi \kappa (\enorm(z) - \langle d-zp, c \rangle))|
  \leq \exp(\kappa |\Im(\enorm(z) - \langle d-zp, c \rangle)|),
\end{equation*}
and Lemma~\ref{le:taylor} suggests that we should find a bound
$C\in\bbbr_{\geq 0}$ satisfying
\begin{align*}
  \kappa \|p\| \frac{r}{2 (\zeta-r)} &\leq C, &
  \kappa \|p\| \left\|\frac{d-tp}{\|d-tp\|} - c\right\|
  &\leq C &
  &\text{ for all } t\in[-1,1],
\end{align*}
with a suitable $r\in[0,\zeta)$.
We first consider the second inequality involving the direction $c$.
Instead of looking for a bound for all $t\in[-1,1]$, we only consider
the centers $m_\tau\in\tau$ and $m_\sigma\in\sigma$ of the two
boxes and require
\begin{equation}\label{eq:adm_directional}
  \kappa \left\| \frac{m_\tau-m_\sigma}{\|m_\tau-m_\sigma\|} - c \right\|
  \leq \frac{\eta_1}{\max\{\diam(\tau),\diam(\sigma)\}}
\end{equation}
for a second admissibility parameter $\eta_1\in\bbbr_{>0}$.
In order to obtain the required estimate, we have to ensure that
$m_\tau-m_\sigma$ is sufficiently close to $d-tp$ by using the condition
\begin{equation}\label{eq:adm_parabolic}
  \kappa \max\{ \diam^2(\tau), \diam^2(\sigma) \}
  \leq \eta_2 \dist(\tau, \sigma).
\end{equation}
\end{subequations}

%
%
\begin{lemma}[Approximate directions]
\label{le:approximate_directions}
Assume that $\tau,\sigma$, and $c$ satisfy the conditions
(\ref{eq:adm_directional}) and (\ref{eq:adm_parabolic}).
Let $d$, $p\in\bbbr^n$ be vectors satisfying (\ref{eq:p_bound}) and
(\ref{eq:dtp_bound}).
Then we have
\begin{align*}
  \left\| \frac{d - t p}{\|d - t p\|} - c \right\|
  &\leq \frac{\eta_1+\eta_2}{2 \kappa \|p\|} &
  &\text{ for all } t\in[-1,1].
\end{align*}
\end{lemma}
\begin{proof}
Let $t\in[-1,1]$ and $q:=\max\{\diam(\tau),\diam(\sigma)\}$.
(\ref{eq:dtp_bound}) and (\ref{eq:adm_parabolic}) yield
\begin{align*}
  \|d - t p\|
  &\geq \dist(\tau, \sigma)
   \geq \frac{\kappa}{\eta_2} q^2, &
  \left\| (d-t p) \frac{\eta_2}{\kappa q^2} \right\| &\geq 1.
\intertext{Due to $m_\tau\in \tau$ and $m_\sigma\in \sigma$, we can apply
(\ref{eq:adm_parabolic}) to find}
  \|m_\tau-m_\sigma\|
  &\geq \dist(\tau, \sigma)
   \geq \frac{\kappa}{\eta_2} q^2, &
  \left\| (m_\tau-m_\sigma) \frac{\eta_2}{\kappa q^2} \right\| &\geq 1.
\end{align*}
Since projecting two points outside the unit ball to its surface does
not increase their distance (cf. \cite[Lemma~7]{BO15}), we obtain
\begin{equation*}
  \left\| \frac{d-t p}{\|d-t p\|} - \frac{m_\tau-m_\sigma}{\|m_\tau-m_\sigma\|} \right\|
  \leq \| (d-t p) - (m_\tau-m_\sigma) \| \frac{\eta_2}{\kappa q^2}.
\end{equation*}
Due to (\ref{eq:dtp_bound}), we can find $x\in \tau$ and $y\in \sigma$
with $d-t p=x-y$ and obtain
\begin{equation*}
  \|(d-t p) - (m_\tau-m_\sigma)\|
  = \|(x-m_\tau) - (y-m_\sigma)\|
  \leq q/2 + q/2 = q.
\end{equation*}
Combining the estimates with (\ref{eq:adm_directional}) yields
\begin{align*}
  \left\| \frac{d-t p}{\|d-t p\|} - c \right\|
  &\leq \left\| \frac{d-t p}{\|d-t p\|}
                - \frac{m_\tau-m_\sigma}{\|m_\tau-m_\sigma\|} \right\|
      + \left\| \frac{m_\tau-m_\sigma}{\|m_\tau-m_\sigma\|} - c \right\|\\
  &\leq \| (d-t p) - (m_\tau-m_\sigma) \| \frac{\eta_2}{\kappa q^2}
      + \frac{\eta_1}{\kappa q}
   \leq \frac{\eta_2+\eta_1}{\kappa q}.
\end{align*}
To complete the proof, we recall that (\ref{eq:p_bound}) yields
$2 \|p\| \leq q$.
\end{proof}

The condition (\ref{eq:adm_directional}) is trivially satisfied by
$c=0$ if $\kappa \max\{\diam(\tau),\diam(\sigma)\} \leq \eta_1$ holds,
i.e., if we are in the low-frequency setting.
Choosing $c=0$ is particularly attractive, since it means that we use
standard polynomial interpolation.  We therefore require
\begin{equation}\label{eq:c=0}
  \kappa \max\{\diam(\tau),\diam(\sigma)\} \leq \eta_1
  \quad\Longrightarrow\quad c = 0.
\end{equation}
We collect our findings in the following definition.

%
%
\begin{definition}[Parabolic admissibility]
A triple $(\tau,\sigma,c)$ satisfies the \emph{parabolic admissibility
condition} if the three conditions (\ref{eq:admissibility}) together
with (\ref{eq:c=0}) hold.
\end{definition}

\subsection{Interpolation error}

The result of Lemma~\ref{le:taylor} provides us with an upper bound
for the exponential term in the numerator of the definition
(\ref{eq:g1_def}) of $g_{dp}$, while Lemma~\ref{le:norm_estimates}
provides us with a lower bound for the denominator.
Since we have assumed stability of the interpolation scheme, we 
only have to prove existence of a good polynomial approximation of $g_{dp}$.  
Following \cite[Chapter~7]{DELO93}, the existence of a holomorphic
extension in a \emph{Bernstein elliptic disc}
\begin{align}\label{eq:bernstein_disc}
  \mathcal{D}_\varrho
  &:= \left\{ z=x+\bi y\ :\ x,y\in\bbbr,
         \left(\frac{2x}{\varrho+1/\varrho}\right)^2
         + \left(\frac{2y}{\varrho-1/\varrho}\right)^2 < 1 \right\} &
  &\text{ for all } \varrho\in\bbbr_{>1}
\end{align}
already implies the existence of such an approximation.

%
%
\begin{lemma}[polynomial approximability]
\label{le:existence}
Let $\hat\varrho\in\bbbr_{>1}$, and let
$f:\mathcal{D}_{\hat\varrho}\to\bbbc$ be holomorphic.
Given $\varrho\in(1,\hat\varrho)$ and $m\in\bbbn$, there is a polynomial
$\pi \in \Pi_ m$ of degree $m$ such that
\begin{equation*}
  \|f-\pi\|_{\infty,[-1,1]}
  \leq \frac{2}{\varrho-1} \varrho^{-m}
        \max\{ |f(z)|\ :\ z\in\overline{\mathcal{D}}_\varrho \}.
\end{equation*}
\end{lemma}
\begin{proof}
This is \cite[eqn.~{(8.7)}, Chap.~7]{DELO93}.
\end{proof}

Our Lemmas~\ref{le:norm_estimates} and \ref{le:taylor} can be used
to obtain bounds for the holomorphic extension of $g_{dp}$ in the
domains $U_r$.
In order to apply Lemma~\ref{le:existence}, we simply have to find
$\varrho>1$ such that $\overline{\mathcal{D}}_\varrho\subseteq U_r$.

%
%
\begin{lemma}[Inclusion]
\label{le:inclusion}
Let $r\in\bbbr_{>0}$, and let $\varrho := \sqrt{r^2+1}+r$.
We have $\overline{\mathcal{D}}_\varrho \subseteq U_r$.
\end{lemma}
\begin{proof}
This is \cite[Lemma~4.77]{BO10}.
\end{proof}

%
%
\begin{theorem}[Approximation of $g_{dp}$]
\label{th:approximation}
Let $c\in\bbbr^3$ and let axis-parallel boxes $\tau$, $\sigma$ satisfy the
admissibility conditions (\ref{eq:admissibility}).
Let $d$, $p\in\bbbr^3$ satisfy (\ref{eq:p_bound}) and (\ref{eq:dtp_bound}).
Set 
\begin{equation}\label{eq:rhohat}
  \hat\varrho := \min\left\{ 2, \frac{3}{2 \eta_2} + 1 \right\}.
\end{equation}
Then there are constants $\Cin\in\bbbr_{\geq 0}$ and $\alpha\in\bbbr_{>1}$
depending only on the admissibility parameters $\eta_1,\eta_2$ and the
stability constants $C_\Lambda,\lambda$
(cf. (\ref{eq:polynomial-growth-Lambda})) such that
\begin{align*}
  \|g_{dp}-\mathfrak{I}[g_{dp}]\|_{\infty,[-1,1]}
  &\leq \frac{\Cin}{4 \pi \dist(\tau,\sigma)} \alpha^{-m/2} \hat\varrho^{-m} &
  &\text{ for all } m\in\bbbn.
\end{align*}
\end{theorem}
\begin{proof}
Due to (\ref{eq:dtp_bound}), we have
\begin{align*}
  \|d-tp\| &\geq \dist(\tau,\sigma) &
  &\text{ for all } t\in[-1,1],
\end{align*}
and therefore
\begin{equation*}
  \zeta = \min\left\{ \frac{\|d-tp\|}{\|p\|}\ :\ t\in[-1,1] \right\}
  \geq \frac{\dist(\tau,\sigma)}{\|p\|}.
\end{equation*}
Combining (\ref{eq:p_bound}) with the standard admissibility condition
(\ref{eq:adm_standard}) and the parabolic admissibility condition
(\ref{eq:adm_parabolic}), we obtain $\zeta \geq 2 / \eta_2$ and
$\zeta \geq 4 \kappa \|p\| / \eta_2$.
Choose 
\begin{equation*}
  r := \min\left\{ 1, \frac{3}{4} \zeta \right\}
    \geq \hat r := \min\left\{ 1, \frac{3}{2 \eta_2} \right\}
\end{equation*}
and consider $z\in U_r$ with $U_r$ defined in (\ref{eq:Ur_def}). 
Lemma~\ref{le:approximate_directions} yields
\begin{align*}
  \left\|\frac{d-tp}{\|d-tp\|} - c\right\|
  &\leq \frac{\eta_1+\eta_2}{2 \kappa \|p\|} &
  &\text{ for all } t\in[-1,1],
\end{align*}
and (\ref{eq:fe_bound}) takes the form
\begin{align*}
  |\Im(\enorm(z) - \langle d-zp, c \rangle)|
  &\leq \|p\| \left( \left\|\frac{d-tp}{\|d-tp\|} - c \right\| r
                   + \frac{1}{2 (\zeta-r)} r^2 \right)\\
  &\leq \|p\| \left( \frac{\eta_1+\eta_2}{2 \kappa \|p\|}
                     + \frac{\eta_2}{8 \kappa \|p\| (1-r/\zeta)} \right)\\
  &= \frac{1}{\kappa} \left( \frac{\eta_1+\eta_2}{2}
                             + \frac{\eta_2}{2} \right)
   \leq \frac{\eta_1+\eta_2}{\kappa}.
\end{align*}
For the denominator of $g_{dp}$, we use Lemma~\ref{le:norm_estimates} to
find
\begin{equation*}
  |\enorm(z)| \geq \|p\| (\zeta-r)
  \geq \dist(\tau,\sigma) (1-r/\zeta)
  \geq \dist(\tau,\sigma) / 4
\end{equation*}
and arrive at
\begin{equation*}
  |g_{dp}(z)|
  \leq \frac{|\exp(\bi \kappa (\enorm(z) - \langle d-zp, c \rangle))|}
            {4 \pi |f(z)|}
  \leq \frac{\exp(\eta_1+\eta_2)}{\pi \dist(\tau,\sigma)}.
\end{equation*}
According to Lemma~\ref{le:inclusion}, $U_r$ contains
$\overline{\mathcal{D}}_\varrho$ for
\begin{equation*}
  \varrho = \sqrt{r^2+1}+r \geq \sqrt{\hat r^2+1}+\hat r
  > \hat r + 1 = \min\left\{ 2, \frac{3}{2 \eta_2} + 1 \right\} = \hat\varrho.
\end{equation*}
Let $\alpha := (\sqrt{\hat r^2+1}+\hat r)/(\hat r+1)$.
We have $\varrho \geq \alpha \hat\rho$, $\alpha>1$, and
(\ref{eq:polynomial-growth-Lambda}) yields that
\begin{equation*}
  \Cin := \sup\left\{ \frac{8 (\Lambda_m+1)}{(\hat\varrho-1) \alpha^{m/2}}
                      \exp(\eta_1+\eta_2)
                 \ :\ m\in\bbbn \right\}
\end{equation*}
is finite.
Let now $m\in\bbbn$.
Lemma~\ref{le:existence} gives us $\pi\in\Pi_m$ with
\begin{align*}
  \|g_{dp} - \pi\|_{\infty,[-1,1]}
  &\leq \frac{2}{\hat\varrho-1} \varrho^{-m}
        \max\{|f(z)|\ :\ z\in\overline{\mathcal{D}}_\varrho \}\\
  &\leq \frac{2}{(\hat\varrho-1) \alpha^{m/2}} \alpha^{-m/2} \hat\varrho^{-m}
        \frac{4 \exp(\eta_1+\eta_2)}{4 \pi \dist(\tau,\sigma)}.
\end{align*}
The well-known best approximation property of interpolation
schemes finally gives 
\begin{equation*}
  \|g_{dp}-\mathfrak{I}[g_{dp}]\|_{\infty,[-1,1]}
  \leq (1+\Lambda_m) \|g_{dp}-\pi\|_{\infty,[-1,1]}
  \leq \frac{\Cin}{4 \pi \dist(\tau,\sigma)} \alpha^{-m/2} \hat\varrho^{-m}.
\end{equation*}
\end{proof}

%
%
\begin{corollary}[Interpolation error for $g$]
\label{co:interpolation_error}
Let $c\in\bbbr^3$ and let the axis-parallel boxes $\tau$, $\sigma$
satisfy the admissibility conditions (\ref{eq:admissibility}).
Let $\hat\varrho$ be given by (\ref{eq:rhohat}). 
Then there is a constant $\Cmi\in\bbbr_{\geq 0}$ depending only on the
admissibility parameters $\eta_1,\eta_2$ and the stability constants
$C_\Lambda,\lambda$ (cf. (\ref{eq:polynomial-growth-Lambda})) such that
\begin{align*}
  \|g-\tilde g_{\tau\sigma}\|_{\infty,\tau\times\sigma}
  &\leq \frac{\Cmi}{4 \pi \dist(\tau,\sigma)}
        \hat\varrho^{-m}
  &\text{ for all } m\in\bbbn.
\end{align*}
\end{corollary}
\begin{proof}
Let $m\in\bbbn$.
We combine Lemma~\ref{le:univariate_formulation} with
Theorem~\ref{th:approximation} to obtain
\begin{equation*}
  \|g - \tilde g_{\tau\sigma}\|_{\infty,\tau\times\sigma}
  \leq 6 \Lambda_m^5 \frac{\Cin}{4 \pi \dist(\tau,\sigma)}
       \alpha^{-m/2} \hat\varrho^{-m}
\end{equation*}
with $\alpha>1$.
Due to the stability assumption (\ref{eq:polynomial-growth-Lambda}),
the supremum
\begin{equation*}
  \Cmi := \sup\left\{ \frac{6 \Lambda_m^5 \Cin}{\alpha^{m/2}}
                       \ :\ m\in\bbbn \right\}
\end{equation*}
is finite and we conclude
\begin{equation*}
  \|g - \tilde g_{\tau\sigma}\|_{\infty,\tau\times\sigma}
  \leq \frac{\Cmi}{4 \pi \dist(\tau,\sigma)} \hat\varrho^{-m}.
\end{equation*}
\end{proof}

%
%
\begin{remark}[Asymptotic rate]
According to \cite[Theorem~8.1, Chap.~7]{DELO93}, we can expect
the error to be bounded by $C_r \varrho^{-m}$ with $\varrho=\sqrt{r^2+1}+r$
for any $r<\zeta$. However, $C_r \rightarrow \infty$ for $r \rightarrow \zeta$ 
is possible.

In the proof of Theorem~\ref{th:approximation}, we have chosen $r$
in a way that leads to a particularly simple estimate for the
exponential term.
\end{remark}

%
%
\section{Two-dimensional Helmholtz kernel}
\label{sec:other-kernels}
The core of the analysis of the 3D case in Section~\ref{sec:g3D} 
is the detailed analysis of the holomorphic extension of the Euclidean
norm, i.e., the function $\enorm$, as it allows for good control
of the functions
$z \mapsto \exp(\bi \kappa (\enorm(z) - \langle d - z p,c\rangle))$.
This opens the door to the analysis of more general kernel functions $k$ for
which (the holomorphic extension of) the ``non-oscillatory'' part
$k(x,y) \exp(-\bi \kappa \|x - y\|)$ can be controlled. 
A particularly interesting case are translation invariant kernel
functions of the form $k(x,y) = k_1(\kappa,\|x - y\|)$, where the map
$z \mapsto k_1(\kappa,z)$ is holomorphic on $\bbbc_+$ and satisfies
suitable conditions there.
The two-dimensional Helmholtz kernel $h$ is such an example. 

%
%
\begin{lemma}
\label{lemma:estimate-hankel}
There exists a $C > 0$ such that
\begin{align*}
  \left| H^{(1)}_0(z) \exp(- \bi z)\right|
  &\leq C \min\left\{1 + \left|\ln |z| \right|, 1/\sqrt{|z|}\right\} &
  &\text{ for all } z\in\bbbc_+.
\end{align*}
\end{lemma}
\begin{proof}
The bound is obtained by studying the cases of small $|z|$ and large
$|z|$ separately.
For small $z$ we use the fact that $H^{(1)}_0(z) = J_0(z) + \bi Y_0(z)$ and
that $J_0$ is analytic with $\lim_{z \rightarrow 0} J_0(z) = 1$
and $Y_0(z) \sim 2/\pi \left( \ln (z/2) + \gamma\right)$
(as $z \rightarrow 0$, $z \in \bbbc_+$) with Euler-Mascheroni's constant
$\gamma$ (cf. \cite[(9.1.12), (9.1.13)]{AS72}) so that  for any $R > 0$
one has a $C_R > 0$ such that
\begin{align*}
  \left| H^{(1)}_0(z) \exp(- \bi \zeta)\right|
  &\leq C_R (1 + |\ln |z||) &
  &\text{ for all } z\in B_R(0) \cap \bbbc_+.
\end{align*}
For large $z$, one uses \cite[Chap.~7, eqns.~(13.2), (13.3)]{OL74} with
$n = 1$ to get
\begin{equation*}
  \left| H^{(1)}_0(z) \exp(- \bi (z - \pi/4))\right|
  \leq \sqrt{ \frac{2}{\pi |z|}}
       \left( 1+ \frac{\pi}{8 |z|} \exp\left(\frac{\pi}{8 |z|}\right)\right).
\end{equation*}
\end{proof}

Using Lemma~\ref{lemma:estimate-hankel} it is possible to formulate the
approximation result corresponding to Theorem~\ref{th:approximation}.

%
%
\begin{lemma}
\label{le:approximation-hankel}
Let $c\in\bbbr^2$ and let axis-parallel boxes $\tau$, $\sigma$
satisfy the admissibility conditions (\ref{eq:admissibility}).
Let $d,p\in\bbbr^2$ be vectors satisfying (\ref{eq:p_bound}) and
(\ref{eq:dtp_bound}).
Let $\hat\varrho$ be given by (\ref{eq:rhohat}). 
Define
\begin{align*}
  h_{dp} \colon [-1,1] &\to \bbbc, &
        t &\mapsto \frac{\bi}{4} H^{(1)}_0(\kappa \|d - t p\|)
                         \exp(-\bi \kappa \langle d-tp,c\rangle).
\end{align*}
Then there are constants $\Cin\in\bbbr_{\geq 0}$ and $\alpha\in\bbbr_{>1}$
depending only on the admissibility parameters $\eta_1,\eta_2$ and the
stability constants $C_\Lambda,\lambda$
(cf. (\ref{eq:polynomial-growth-Lambda})) such that for all $m \in \bbbn$
\begin{align*}
  \|h_{dp}-\mathfrak{I}[h_{dp}]\|_{\infty,[-1,1]}
  &\leq \Cin \min\{ 1 + |\ln (\kappa \dist(\tau,\sigma))|,
                    \,(\kappa \dist(\tau,\sigma))^{-1/2}\}
        \alpha^{-m/2} \hat\varrho^{-m}.
\end{align*}
\end{lemma}
\begin{proof}
The key is to recall that $z \mapsto \enorm(z)$ is the holomorphic extension of 
$t \mapsto \|d - tp\|$ so that the analog of the univariate function $g_{dp}$ reads 
\begin{align*}
  h_{dp}(z)
  &= \frac{\bi}{4} \underbrace{ H^{(1)}_0(\kappa \enorm(z))
                                \exp(-\bi\kappa \enorm(z)) }_{=:A(z)}
     \underbrace{ \exp(\bi\kappa \left( \enorm(z)
                     - \langle d - zp,c\rangle\right)}_{=:B(z)}.
\end{align*}
Following the proof of Theorem~\ref{th:approximation} we have to control $h_{dp}$ on $U_r$ (with $r$ given
in the proof of Theorem~\ref{th:approximation}). By the proof of Theorem~\ref{th:approximation}
we have for $z \in U_r$ that $\enorm(z) \in \bbbc_+$ and that
$|\enorm(z)| \ge \dist(\tau,\sigma)/4$. We conclude with Lemma~\ref{lemma:estimate-hankel}
that
$$
|A(z)| \leq C \min\{1 + |\ln (\kappa \operatorname*{dist}(\sigma,\tau))|, (\kappa \operatorname*{dist}(\sigma,\tau))^{-1/2}\}. 
$$
The term $B(z)$ is estimated in the proof of Theorem~\ref{th:approximation} by
$|B(z)| \leq \exp(\eta_1 + \eta_2)$. The result follows as in Theorem~\ref{th:approximation}.
\end{proof}

Reasoning as in the proof of Corollary~\ref{co:interpolation_error} we arrive
at the following result.

%
%
\begin{corollary}[Interpolation error for $h$]
\label{co:interpolation_error_2D}
Let $c\in\bbbr^2$ with $\|c\|=1$ and let axis-parallel boxes $\tau$, $\sigma$
satisfy the admissibility conditions (\ref{eq:admissibility}).
Let $\hat\varrho$ be given by (\ref{eq:rhohat}). 
Then there is a constant $\Cmi\in\bbbr_{\geq 0}$ depending only on the
admissibility parameters $\eta_1,\eta_2$ and the stability constants
$C_\Lambda,\lambda$ (cf. (\ref{eq:polynomial-growth-Lambda})) such that for 
all $m \in \bbbn$ 
\begin{align*}
  \|h-\tilde h_{\tau\sigma}\|_{\infty,\tau\times\sigma}
  &\leq \Cmi \min\{ 1 + |\ln (\kappa \dist(\tau,\sigma))|,
                   \,(\kappa \dist(\tau,\sigma))^{-1/2}\} \hat\varrho^{-m}. 
\end{align*}
\end{corollary}

%
%
\section{Nested approximation}
\label{sec:nested_approximation}

As mentioned before, the crux of polylogarithmic-linear complexity
algorithms is a nested multilevel structure.
The vital ingredient that permits this structure is the approximation
step (\ref{eq:nested_interpolation}).
In this section, we analyze the impact of this step.
Structurally similar analyses can be found in \cite{BOLOME02,SA00,BOBOME17}.

\subsection{Reduction to univariate nested interpolation}

We recall the setting of Section~\ref{sec:error-analysis-via-multilevel}:
We are given sequences of axis-parallel boxes 
\begin{align}\label{eq:cluster_sequences}
  \tau_0 &\supseteq \tau_1 \supseteq \cdots \supseteq \tau_L, &
  \sigma_0 &\supseteq \sigma_1 \supseteq \cdots \supseteq \sigma_L
\end{align}
and a sequence $c_0,\ldots,c_L\in\bbbr^n$ of directions.
We are interested in the directional interpolation operators
$\mathfrak{I}_{\tau_\ell,c_\ell}$ given by
\begin{align}
\label{eq:I_ell_def}
  \mathfrak{I}_{\tau_\ell,c_\ell}[u]
  &= \exp(\bi \kappa \langle c_\ell, \punkt \rangle)
     \mathfrak{I}_{\tau_\ell}
     [\exp(-\bi \kappa \langle c_\ell, \punkt \rangle)u] &
  &\text{ for all } u\in C(\tau_\ell),
\end{align}
and similar operators $\mathfrak{I}_{\sigma_\ell,-c_\ell}$ for the source
clusters $\sigma_\ell$.
With the aid of these operators, we write the operators
$\mathfrak{I}_{\tau_\ell \times \sigma_\ell,c_\ell}$
of (\ref{eq:directional-interpolation-operators}) as tensor product
operators
\begin{align*}
  \mathfrak{I}_{\tau_\ell \times  \sigma_\ell,c_\ell}
  &= \mathfrak{I}_{\tau_\ell,c_\ell} \otimes
     \mathfrak{I}_{\sigma_\ell,-c_\ell} &
  &\text{ for all } \ell\in[0:L]
\end{align*}
and approximate the kernel function $k$ by
\begin{equation*}
  \tilde k_{\tau\sigma}
  = \mathfrak{I}_{\tau_L\times\sigma_L,c_L}\circ\cdots
    \circ\mathfrak{I}_{\tau_0\times\sigma_0,c_0}[k].
\end{equation*}
In order to estimate the approximation error, we can rely
on (\ref{eq:nested-interpolation-error-vorne}) to find that we
only need a stability estimate for the iterated operators, since
we already have Corollaries~\ref{co:interpolation_error} and
\ref{co:interpolation_error_2D} at our disposal.
The iterated operators can be rewritten as 
\begin{align*}
  \mathfrak{I}_{\tau_L\times\sigma_L,c_L} \circ
    \cdots \circ \mathfrak{I}_{\tau_{\ell+1}\times\sigma_{\ell+1},c_{\ell+1}} 
  &\!=\! \Bigl( \mathfrak{I}_{\tau_L,c_L} \circ \cdots
             \circ \mathfrak{I}_{\tau_{\ell+1},c_{\ell+1}}\Bigr)\!
\otimes \! \Bigl( \mathfrak{I}_{\sigma_L,-c_L} \circ \cdots
             \circ \mathfrak{I}_{\sigma_{\ell+1},-c_{\ell+1}}\Bigr).
\end{align*} 
Since 
\begin{align}
\label{eq:nested-interpolation-error-use-of-norm} 
& \| \mathfrak{I}_{\tau_L\times\sigma_L,c_L} \circ \cdots
     \circ  \mathfrak{I}_{\tau_{\ell+1}\times\sigma_{\ell+1},c_{\ell+1}}
  \|_{\infty,\tau_L\times\sigma_L\leftarrow
    \tau_{\ell+1}\times\sigma_{\ell+1}}\\
  &\qquad\leq
     \| \mathfrak{I}_{\tau_L,c_L} \circ \cdots
         \circ \mathfrak{I}_{\tau_{\ell+1},c_{\ell+1}} 
     \|_{\infty,\tau_L\leftarrow\tau_{\ell+1}}
     \| \mathfrak{I}_{\sigma_L,-c_L} \circ \cdots
         \circ \mathfrak{I}_{\sigma_{\ell+1},-c_{\ell+1}} 
     \|_{\infty,\sigma_L\leftarrow\sigma_{\ell+1}}\notag, 
\end{align}
we have reduced the quest for the stability estimates required by
(\ref{eq:nested-interpolation-error-vorne}) to a stability analysis
of the operators 
$\mathfrak{I}_{\tau_L,c_L} \circ \cdots \circ
 \mathfrak{I}_{\tau_{\ell+1},c_{\ell+1}}$
and 
$\mathfrak{I}_{\sigma_L,-c_L} \circ \cdots \circ
 \mathfrak{I}_{\sigma_{\ell+1},-c_{\ell+1}}$. 
Their stability properties depend on how quickly the boxes shrink
and how small the differences $\|c_{\ell+1} - c_\ell\|$ are.
Our final result is recorded in Theorem~\ref{thm:nested_directional}. 

%
%
\begin{remark}
\label{rem:mathfrakI-vs-widetildemathfrakI}
We have $\mathfrak{I}_{\tau_\ell,0}=\mathfrak{I}_{\tau_\ell}$ and
$\mathfrak{I}_{\sigma_\ell,0}=\mathfrak{I}_{\sigma_\ell}$.

If $u\in C(\sigma_\ell)$ is real-valued, we have
$\mathfrak{I}_{\sigma_\ell,-c_\ell}[u]=\overline{\mathfrak{I}_{\sigma_\ell,c_\ell}[u]}$.
\end{remark}

%
%
\begin{remark}[Re-interpolated Lagrange polynomials]
In view of Section~\ref{sec:error-analysis-via-multilevel},
we can reduce the error analysis to estimating
$L_{\tau c,\nu} - \widetilde{L}_{\tau c,\nu}$.
We expect this approach to be slightly sharper than
resorting to the rather general bounds 
(\ref{eq:nested-interpolation-error-vorne}) and 
(\ref{eq:nested-interpolation-error-use-of-norm}). 
\end{remark}

Since the operators $\mathfrak{I}_{\tau_\ell,c_\ell}$ have product structure,
their analysis can be broken down further to that of understanding
operators acting on univariate functions.
Specifically, writing the axis-parallel boxes $\tau_\ell$ in the form 
\begin{equation}
\label{eq:boxes}
  \tau_\ell =
  [a_{\ell,1},b_{\ell,1}]\times\cdots\times[a_{\ell,n},b_{\ell,n}]
\end{equation}
and observing $\exp(\bi \kappa \langle c, x \rangle)
= \exp(\bi \kappa c_1 x_1) \cdots \exp(\bi \kappa c_n x_n)$,
we have
\begin{align*}
  \mathfrak{I}_{\tau_\ell,c_\ell}[u]
  &= \exp(\bi \kappa \langle c_\ell, \punkt \rangle)
     \mathfrak{I}_{[a_{\ell,1},b_{\ell,1}]}\otimes\cdots
       \otimes\mathfrak{I}_{[a_{\ell,n},b_{\ell,n}]}
     [\exp(-\bi \kappa \langle c_\ell, \punkt \rangle) u]\\
  &= \mathfrak{I}_{\ell,1}\otimes\cdots\otimes\mathfrak{I}_{\ell,n}[u]
    \qquad\text{ for all } u\in C(\tau_\ell)
\end{align*}
with the one-dimensional interpolation operators
\begin{align}\label{eq:I_ell_iota_def}
  \mathfrak{I}_{\ell,\iota} \colon
    C[a_{\ell,\iota},b_{\ell,\iota}] &\to C[a_{\ell,\iota},b_{\ell,\iota}], &
  u &\mapsto \exp(\bi \kappa c_{\ell,\iota} \punkt)
       \mathfrak{I}_{[a_{\ell,\iota},b_{\ell,\iota}]}
       [\exp(-\bi \kappa c_{\ell,\iota} \punkt) u]
\end{align}
for all $\ell\in[0:L]$ and $\iota\in[1:n]$.
We note 
\begin{equation}\label{eq:reinterpolation_reduced}
  \mathfrak{I}_{\tau_L,c_L} \circ \cdots
     \circ \mathfrak{I}_{\tau_\ell+1,c_{\ell+1}}
   = (\mathfrak{I}_{L,1} \circ \cdots \mathfrak{I}_{\ell+1,1}) \otimes
     \cdots
     \otimes (\mathfrak{I}_{L,n} \circ \cdots \mathfrak{I}_{\ell+1,n}),  
\end{equation}
so that the stability analysis of 
$\mathfrak{I}_{\tau_L,c_L} \circ \cdots \circ
\mathfrak{I}_{\tau_{\ell+1},c_{\ell+1}}$
is indeed reduced to that  of the operators  
$(\mathfrak{I}_{L,\iota} \circ \cdots \mathfrak{I}_{\ell+1,\iota})$
for $\iota \in [1:n]$. 
\subsection{Recursive reinterpolation in 1D}

Let $\mathfrak{C} := (J_\ell)_{\ell=0}^L$ be a tuple of non-empty intervals
\begin{equation}
\label{eq:nested-sequence-1D}
  J_0 \supseteq J_1 \supseteq J_2 \supseteq \cdots \supseteq J_L.
\end{equation}
We assume that there is a \emph{contraction factor} $\bar q\in\bbbr$
such that
\begin{align}\label{eq:contraction_factor}
  \frac{|J_\ell|}{|J_{\ell-1}|} &\leq \bar q < 1 &
  &\text{ for all } \ell\in[1:L].
\end{align}
We fix $c_0,\ldots,c_L\in\bbbr$ and denote the weighted
interpolation operators by
\begin{align}
\label{eq:frak_I_ell}
  \mathfrak{I}_\ell \colon C(J_\ell) &\to C(J_\ell), &
    u &\mapsto \exp(\bi \kappa c_\ell \punkt)
          \mathfrak{I}_{J_\ell}[\exp(-\bi \kappa c_\ell \punkt) u] &
  &\text{ for all } \ell\in[0:L].
\end{align}
The iterated interpolation operator is given by
\begin{equation*}
  \mathfrak{I}_{\mathfrak{C}}
  := \mathfrak{I}_L \circ \cdots \circ \mathfrak{I}_0.
\end{equation*}
The stability analysis of $\mathfrak{I}_{\mathfrak{C}}$ 
uses the Bernstein estimate to bound a polynomial
in a Bernstein disc $\mathcal{D}_\alpha$ and then applies
Lemma~\ref{le:existence} to find an approximation in a sub-interval.
For the latter approximation step, we need the following geometrical result.

%
%
\begin{lemma}[Inclusion]
\label{lemma:lazy-ellipse-inclusions}
Let $ -1 \leq a < b \leq 1$ and $h:= (b-a)/2$.
For $\alpha > 1$ denote by
\begin{equation*}
{\mathcal D}_{\alpha}^{a,b}:= \Phi_{[a,b]}(\mathcal{D}_\alpha)
\end{equation*}
the transformed Bernstein disc $\mathcal{D}_\alpha$
(cf. (\ref{eq:bernstein_disc})) for the interval $[a,b]$.
Fix $ \varepsilon \in (0,1)$. Then there is a $\varrho_0 > 1$ (depending
solely on $\varepsilon$) such that 
\begin{align*}
  {\mathcal D}^{a,b}_{(1-\varepsilon) \varrho/h}
  &\subset {\mathcal D}_\varrho &
  &\text{ for all } \varrho\geq\varrho_0.
\end{align*}
\end{lemma}
\begin{proof}
We exploit that for large $\varrho$ the Bernstein disc 
${\mathcal D}_\varrho$ is essentially a (classical) disc of radius $\varrho/2$. 
We start from the following inclusion of discs in Bernstein elliptic discs 
and \emph{vice versa}: 
\begin{equation*}
B_{(\varrho - 1/\varrho)/2}(0) \subset {\mathcal D}_{\varrho} \subset B_{(\varrho +1/\varrho)/2}(0),
\end{equation*}
where $B_r(x)=\{|z-x|\leq r\, \colon\,  z\in\bbbc\}$ denotes the closed disc
around $x$ of radius $r$.
Hence, we have to show (for $\varrho$ sufficiently large) that for $\alpha = (1-\varepsilon) \varrho/h$ we have 
\begin{align*}
{\mathcal D}_{\alpha}^{a,b}  &= \frac{a+b}{2} + h {\mathcal D}_\alpha 
\subset \frac{a+b}{2} + h B_{(\alpha +1/\alpha)/2}(0)\\
& = B_{h (\alpha+1/\alpha)/2}\left(\frac{a+b}{2}\right)
\stackrel{!}{\subset} B_{(\varrho -1/\varrho)/2}(0);  
\end{align*}
all inclusions are geometrically clear with the exception of the last one. 
To ensure that one, we require
\begin{equation}
\label{eq:lemma:lazy-ellipse-inclusions-100}
1 + h \frac{\alpha + 1/\alpha}{2} \leq \frac{\varrho - 1/\varrho}{2}. 
\end{equation}
Inserting the condition $\alpha = (1 - \varepsilon) \varrho/h$ and rearranging terms, we 
see that (\ref{eq:lemma:lazy-ellipse-inclusions-100}) is true if we ensure 
\begin{equation}
\label{eq:lemma:lazy-ellipse-inclusions-200}
\varepsilon \varrho^2 \ge 2 \varrho + 1 + \frac{h^2}{1-\varepsilon}. 
\end{equation}
In view of $h \in [0,2]$, this last condition is certainly met if 
\begin{equation*}
  \varrho \ge \varrho_0
  := \frac{1 + \sqrt{1 + \varepsilon (1 + 4/(1-\varepsilon))}}{\varepsilon}.  
\end{equation*}
\end{proof}

%
%
\begin{lemma}
\label{lemma:approximation-with-exponential-factor}
Fix $q \in (\bar{q},1)$ and $\gamma > 0$. 
Then there is $m_0 > 0$ depending only on $\bar q$, $\gamma$, and $q$ 
such that the following is true: 

Let $J_1 \subset J_0$ be two closed intervals with
$|J_1|/|J_0| \leq \bar q  < 1$. 
Denote $h_0:= |J_0|/2$, $h_1:= |J_1|/2$.  
Let $\kappa$, $c_0$, $c_1 \in {\mathbb R}$ and assume that 
\begin{equation*}
  |\kappa h_0 (c_0 - c_1)| \leq \gamma. 
\end{equation*}
Then for all $m \ge m_0$ and all $\pi \in \Pi_m$ 
\begin{equation*}
  \inf_{v \in \Pi_m} \|\exp(\bi \kappa  c_0 \punkt) \pi
                    - \exp(\bi \kappa c_1 \punkt) v\|_{\infty,J_1} 
  \leq q^m \|\pi\|_{\infty,J_0}.
\end{equation*}
\end{lemma}
\begin{proof}
Let $\Phi := \Phi_{J_0} \colon [-1,1] \rightarrow J_0$ be the orientation
preserving affine bijection as in Section~\ref{sec:tensor_interpolation}. 
Set $\widehat \pi:= \pi \circ \Phi$, $[a,b]:= \widehat J_1:= \Phi^{-1}(J_1)$, 
$h:= h_1/h_0 = (b-a)/2 \leq \overline{q}$. 
We have 
\begin{equation*}
\inf_{v \in \Pi_m} \|\exp(\bi \kappa c_0 \punkt) \pi
                   - \exp(\bi \kappa c_1 \punkt) v\|_{\infty,J_1} 
 = \inf_{v \in \Pi_m} \|\exp(\bi \kappa  h_0 (c_0 - c_1)\punkt) \widehat \pi
                      - v\|_{\infty,\widehat J_1}. 
\end{equation*}
By the polynomial approximation results of Lemma~\ref{le:existence}, we
estimate for arbitrary $\alpha > 1$ and $m \in {\mathbb N}_0$ 
\begin{align*}
  \inf_{v \in \Pi_m}
  &\|\exp(\bi \kappa  h_0  (c_0 - c_1)\punkt) \widehat \pi
      - v\|_{\infty,\widehat J_1}
   \leq \frac{2\alpha^{-m} }{\alpha-1} 
      \| \exp(\bi \kappa  h_0 (c_0 - c_1)\punkt)
         \widehat \pi\|_{\infty,{\mathcal D}^{a,b}_{\alpha}}\\
  &\leq \frac{2\alpha^{-m}}{\alpha-1} 
      \exp\left(|\kappa  h_0 (c_0 - c_1)| h
                   \frac{\alpha -1/\alpha}{2}\right)
      \|\widehat \pi\|_{\infty,{\mathcal D}^{a,b}_{\alpha}}. 
\end{align*}
We now choose $\alpha$ in dependence on $m$.
Fix $\varepsilon \in (0,1-\bar q/q)$
(so that $\bar q/(1-\varepsilon) < q$) and choose $\beta > 0$ such that
(for the $q$ of the statement of the lemma)
\begin{equation*}
  q = \frac{\bar q}{1-\varepsilon}
      \exp\left( \frac{\gamma (1-\varepsilon) \beta}{2} \right).
\end{equation*}
We set $\varrho = \beta m$ and
$\alpha = (1-\varepsilon) \varrho/\bar q = (1-\varepsilon)\beta m/\bar q$. 
Lemma~\ref{lemma:lazy-ellipse-inclusions} implies 
${\mathcal D}^{a,b}_{\alpha} \subset {\mathcal D}_\varrho$
if $\beta m = \varrho \ge \varrho_0$.
We note that this condition imposes $m \ge \varrho_0/\beta$. 
Furthermore, the Bernstein estimate \cite[Thm.~{2.2}, Chap.~4]{DELO93}, gives 
\begin{equation*}
  \|\widehat \pi\|_{\infty,{\mathcal D}^{a,b}_{\alpha}} 
  \leq \|\widehat \pi\|_{\infty,{\mathcal D}_{\varrho}} 
  \leq \varrho^m \|\widehat \pi\|_{\infty,(-1,1)}. 
\end{equation*}
Hence we obtain 
\begin{align*}
  \inf_{v \in \Pi_m} & \|\exp(\bi \kappa  h_0  (c_0 - c_1)\punkt) \widehat \pi
                      - v\|_{\infty,\widehat J_1}
   \leq \frac{2}{\alpha-1} \left(\frac{\varrho}{\alpha}\right)^{m} 
          \exp\left(\gamma h \frac{\alpha -1/\alpha}{2}\right)
          \|\widehat \pi\|_{\infty,(-1,1)}\\
  &\leq \frac{2}{\alpha-1} \left(\frac{\bar q}{1-\varepsilon}\right)^{m} 
          \exp\left(\frac{\gamma \bar q \alpha}{2}\right)
          \|\widehat \pi\|_{\infty,(-1,1)}\\
  &= \frac{2}{\alpha-1} \left(\frac{\bar q}{1-\varepsilon}\right)^{m}  
          \exp\left(m\frac{\gamma (1-\varepsilon) \beta}{2}\right)
          \|\widehat \pi\|_{\infty,(-1,1)}\\
  &= \frac{2}{\alpha-1} \left( \frac{\bar q}{1-\varepsilon}
          \exp\left(\frac{\gamma (1-\varepsilon) \beta}{2}\right) \right)^m
          \|\widehat \pi\|_{\infty,(-1,1)}
   =  \frac{2}{\alpha-1} q^m \|\widehat \pi\|_{\infty,(-1,1)}\\
  &\leq q^m \|\widehat \pi\|_{\infty,(-1,1)},  
\end{align*}
where, in the  last step we used that $\alpha \to \infty$ as
$m \to \infty$; more precisely, we ensure $\alpha \geq 3$ by requiring 
\begin{equation*}
  m \ge m_0 := \max\left\{ \left\lceil \frac{\varrho_0}{\beta}\right\rceil,
    \left\lceil 3 \frac{\bar q}{(1-\varepsilon) \beta} \right\rceil
    \right\}.
\end{equation*}
\end{proof}

%
%
\begin{lemma}[Interpolation error]
\label{le:stability-II-i}
Let $\ell\in[1:L]$ and $\pi\in\Pi_m$.
We have
\begin{align}
\label{eq:stability-II-i}
  \|\exp(\bi \kappa  c_{\ell-1} \punkt) \pi
  &- \mathfrak{I}_\ell [\exp(\bi \kappa  c_{\ell-1} \punkt)
    \pi]\|_{\infty,J_\ell}\notag\\
  &\leq (1 + \Lambda_m)
     \inf_{v \in \Pi_m} \|\exp(\bi \kappa  c_{\ell-1} \punkt) \pi
                       - \exp(\bi \kappa  c_\ell \punkt) v\|_{\infty,J_\ell}. 
\end{align}
\end{lemma}
\begin{proof}
Let $v\in\Pi_m$ be arbitrary.
Write 
\begin{align*}
  \exp(\bi \kappa  c_{\ell-1} \punkt ) \pi
  &- \mathfrak{I}_\ell [\exp(\bi \kappa  c_{\ell-1} \punkt) \pi]\\
  &= \exp(\bi \kappa  c_{\ell-1} \punkt) \pi
     - \exp(\bi \kappa  c_\ell \punkt ) v\\
  &\quad -\exp(\bi \kappa  c_\ell \punkt) {\mathfrak I}_{J_\ell}
       [\exp(-\bi \kappa  c_\ell \punkt)
          \{ \exp(\bi \kappa  c_{\ell-1} \punkt) \pi
             - \exp(\bi \kappa  c_\ell \punkt) v\}]. 
\end{align*}
Hence, by the stability of the polynomial interpolation operator ${\mathfrak I}$ we get 
\begin{align*}
  \|\exp(\bi \kappa  c_{\ell-1} \punkt) \pi
  &- \mathfrak{I}_\ell[ \exp(\bi \kappa  c_{i-1} \punkt) \pi] \|_{\infty,J_\ell}\\
  &\leq (1 + \Lambda_m) 
   \|\exp(\bi \kappa  c_{\ell-1} \punkt) \pi
     - \exp(\bi \kappa  c_\ell \punkt) v \|_{\infty,J_\ell}.
\end{align*}
\end{proof}

%
%
\begin{theorem}[Stability of reinterpolation]
\label{th:iterated-stability-neu}
Let the tuple $(J_\ell)_{\ell=0}^L$ satisfy (\ref{eq:nested-sequence-1D}) and (\ref{eq:contraction_factor}). 
Write $h_\ell = |J_\ell|/2$ for all $\ell\in[0:L]$.
Let $c_0,\ldots,c_L \in {\mathbb R}$ be such that, for some $\gamma\ge 0$,  
\begin{align}\label{eq:lemma:iterated-stability-neu-10}
  |\kappa  h_{\ell-1} (c_{\ell-1} - c_\ell)| &\leq \gamma &
  &\text{ for all } \ell\in[1:L].
\end{align}
Let the operators ${\mathfrak I}_\ell$, $\ell\in[0:L]$ 
be given by (\ref{eq:frak_I_ell}). 
Fix $q \in (\bar q,1)$.
Then there is $m_0 > 0$ depending solely on $\gamma$, $\bar q$, and
$q$, such that for all $m \ge m_0$
\begin{subequations}
\begin{align}
  \|(I - \mathfrak{I}_L\circ\cdots\circ\mathfrak{I}_1)
     [\exp(\bi \kappa c_0 \punkt ) \pi]\|_{\infty,J_L}
  &\leq \varepsilon_{m,L} \|\exp(\bi \kappa  c_0 \punkt) \pi\|_{\infty,J_0} &
  &\text{ for all } \pi\in\Pi_m,\label{eq:lemma:iterated-stability-neu-12}\\
  \|\mathfrak{I}_{\mathfrak C}\|_{C(J_L)\leftarrow C(J_0)}
  &\leq \Lambda_m \left( 1 + \varepsilon_{m,L} \right), 
  \label{eq:lemma:iterated-stability-neu-15}\\
  \varepsilon_{m,L} &:= (1 + (1+\Lambda_m) q^m)^L - 1.
  \label{eq:lemma:iterated-stability-neu-17}
\end{align}
\end{subequations}
Choose $\widehat q \in (q,1)$.
Then there is $K>0$ depending solely on $\gamma$, $\bar q$, the chosen
$\widehat q$, and the constants $C_\Lambda$, $\lambda$ of
(\ref{eq:polynomial-growth-Lambda}), such that the following implication
holds: 
\begin{equation}\label{eq:lemma:iterated-stability-neu-20}
  m \ge K (1 + \log L)
  \qquad \Longrightarrow \qquad \varepsilon_{m,L} \leq \widehat q^m. 
\end{equation}
\end{theorem}
\begin{proof}
Let $m_0$ be given by Lemma~\ref{lemma:approximation-with-exponential-factor},
and assume $m\geq m_0$.

\emph{Step 1.} (stability of $\mathfrak{I}_\ell$).
Combining 
Lemmas~\ref{lemma:approximation-with-exponential-factor} and
\ref{le:stability-II-i}, the following estimate
holds for arbitrary $\pi \in \Pi_m$ and $\ell\in[1:L]$:
\begin{subequations}
\begin{equation}\label{eq:lemma:iterated-stability-neu-100}
  \|\exp(\bi \kappa  c_{\ell-1} \punkt) \pi
    - \mathfrak{I}_\ell [\exp(\bi \kappa c_{\ell-1} \punkt)
      \pi]\|_{\infty,J_\ell}
  \leq (1 + \Lambda_m ) q^m
     \|\exp(\bi \kappa c_{\ell-1} \punkt) \pi\|_{\infty,J_{\ell-1}}.
\end{equation}
The triangle inequality yields the stability estimate
\begin{equation}\label{eq:lemma:iterated-stability-neu-110}
  \|\mathfrak{I}_\ell [\exp(\bi \kappa c_{\ell-1} \punkt) \pi]\|_{\infty,J_\ell}
  \leq (1 + (1+\Lambda_m )q^m)
    \|\exp(\bi \kappa c_{\ell-1} \punkt) \pi\|_{\infty,J_{\ell-1}}.  
\end{equation}
\end{subequations}

\emph{Step 2.}  (error estimate)
We note the following telescoping sum for $\ell=1,\ldots,L$:
\begin{align}
\label{eq:thm:nested-interpolation-80}
  E_\ell &:= I - \mathfrak{I}_\ell \circ \cdots \circ \mathfrak{I}_1\\
\nonumber 
  &= (I - \mathfrak{I}_1)
     + (I - \mathfrak{I}_2) \circ \mathfrak{I}_1
     + (I - \mathfrak{I}_3) \circ \mathfrak{I}_2 \circ \mathfrak{I}_1
     + \cdots
+ (I - \mathfrak{I}_\ell) \circ \mathfrak{I}_{\ell-1}
             \circ \cdots \circ \mathfrak{I}_1.
\end{align}
We claim the following estimates for $\ell \in [1:L]$:
\begin{align} 
\label{eq:thm:nested-interpolation-100}
\|E_\ell \left[\exp(\bi \kappa c_0 \punkt) \pi\right] \|_{\infty,J_\ell} 
&\leq \varepsilon_{m,\ell}  \|\pi\|_{\infty,J_0}, \\ 
\label{eq:thm:nested-interpolation-200}
\|{\mathfrak I}_\ell \circ \cdots \circ {\mathfrak I}_1 \left[\exp(\bi \kappa c_0 \punkt) \pi\right] \|_{\infty,J_{\ell}} 
&\leq (1+ \varepsilon_{m,\ell}) \|\pi\|_{\infty,J_0}. 
\end{align}
This is proved by induction on $\ell$. For $\ell = 1$, the estimate (\ref{eq:thm:nested-interpolation-100})
expresses (\ref{eq:lemma:iterated-stability-neu-100}), and 
(\ref{eq:thm:nested-interpolation-200}) then follows from the observation 
${\mathfrak I}_1 = I - E_1$ and the triangle inequality. 
To complete the induction argument,
assume that
(\ref{eq:thm:nested-interpolation-100}), (\ref{eq:thm:nested-interpolation-200}) are proven up to $\ell-1$.
We note that ${\mathfrak I}_{i} \circ \cdots \circ {\mathfrak I}_1 \left[\exp(\bi \kappa c_0 \punkt) \pi\right] = 
\exp(\bi \kappa c_i \punkt)  \widetilde \pi$ for some $\widetilde \pi \in \Pi_m$ for every $i$. 
Therefore, the induction hypothesis and (\ref{eq:thm:nested-interpolation-200}) imply for $i=1,\ldots, \ell-1$
\begin{align}
\label{eq:thm:nested-interpolation-2000} 
& \| (I - {\mathfrak I}_{i+1}) {\mathfrak I}_{i} \circ \cdots \circ {\mathfrak I}_1\left[\exp(\bi \kappa c_0 \punkt) \pi\right]\|_{\infty,J_{i+1}}  
= \| (I - {\mathfrak I}_{i+1}) \left[\exp(\bi \kappa c_i \punkt) \widetilde \pi\right]\|_{\infty,J_{i+1}}    \\
\nonumber 
& \quad 
\stackrel{(\ref{eq:lemma:iterated-stability-neu-100})}{\leq}\!\! 
(1 + \Lambda_m) q^m \|\exp(\bi \kappa c_i \punkt) \widetilde \pi\|_{\infty,J_i} 
 = (1+\Lambda_m)  q^m \|
({\mathfrak I}_{i} \circ \cdots \circ {\mathfrak I}_1)\! \left[\exp(\bi \kappa c_0 \punkt)  \pi\right]\! \|_{\infty,J_i}. 
\end{align}
Next, we get from
(\ref{eq:thm:nested-interpolation-80}),
(\ref{eq:thm:nested-interpolation-200}),
(\ref{eq:thm:nested-interpolation-2000}),
and the geometric series
\begin{align*}
\|E_\ell \left[\exp(\bi \kappa c_0 \punkt)  \pi\right]\|_{\infty,J_{\ell}} &
\stackrel{ (\ref{eq:thm:nested-interpolation-80})}{ \leq }
\sum_{i=0}^{\ell -1} 
\|(I - {\mathfrak I}_{i+1}) ({\mathfrak I}_i \circ \cdots \circ {\mathfrak I}_1) \left[\exp(\bi \kappa c_0 \punkt) \pi\right]\|_{\infty,J_\ell} \\
& \stackrel{(\ref{eq:thm:nested-interpolation-2000})}{\leq}
\sum_{i=0}^{\ell -1} (1+\Lambda_m) q^m 
\|({\mathfrak I}_i \circ \cdots \circ {\mathfrak I}_1) \left[\exp(\bi \kappa c_0 \punkt)  \pi\right]\|_{\infty,J_i} \\
& \stackrel{(\ref{eq:thm:nested-interpolation-200})}{\leq} \sum_{i=0}^{\ell-1} 
(1 + \Lambda_m)  q^m (1 + \varepsilon_{m,i}) \|\pi\|_{\infty,J_0} \\
& = (1+\Lambda_m) q^m \frac{ (1+(1 + \Lambda_m) q^m))^\ell- 1}{(1+(1+\Lambda_m) q^m)-1} \|\pi\|_{\infty,J_0} 
 = \varepsilon_{m,\ell} \|\pi\|_{\infty,J_0}, 
\end{align*}
which is the desired induction step for (\ref{eq:thm:nested-interpolation-100}). The induction step
for (\ref{eq:thm:nested-interpolation-200}) is now a simple application of the triangle inequality.

\emph{Step 3.} (stability estimate)
We consider $u\in C(J_0)$ and define
$\pi_0\in\Pi_m$ by $\pi_0:={\mathfrak I}_{J_0} [\exp(-\bi \kappa c_0 \punkt) u]$.
By definition of $\mathfrak{I}_0$, we have
\begin{gather*}
  \mathfrak{I}_0[u] = \exp(\bi \kappa c_0 \punkt) \pi_0,\qquad
  \|\pi_0\|_{\infty,J_0} \leq \Lambda_m \|u\|_{\infty,J_0},\\
  \mathfrak{I}_{\mathfrak{C}}[u] = \mathfrak{I}_L \circ \cdots
      \circ\mathfrak{I}_1[\exp(\bi \kappa c_0 \punkt) \pi_0].
\end{gather*}
Therefore, 
\begin{align*}
  \|\mathfrak{I}_{\mathfrak{C}}[u]\|_{\infty,J_L}
\stackrel{(\ref{eq:thm:nested-interpolation-200})}{\leq} 
(1 + \varepsilon_{m,L}) \|\exp(\bi \kappa c_0 \punkt) \pi_0\|_{\infty,J_0} 
\leq (1 + \varepsilon_{m,L}) \Lambda_m \|u\|_{\infty,J_0}, 
\end{align*}
which is (\ref{eq:lemma:iterated-stability-neu-15}). 

\emph{Step 4.} (bound for $\varepsilon_{m,L}$)
Let $\widetilde{q}\in(q,\widehat{q})$.
The stability assumption (\ref{eq:polynomial-growth-Lambda}) implies
that we can find $m_1$ such that
\begin{align*}
  (1 + \Lambda_m) q^m
  &\leq (1 + C_\Lambda (m+1)^\lambda) q^m
   \leq \widetilde q^m &
  &\text{ for all } m\geq m_1.
\end{align*}
Hence, we obtain
\begin{equation*}
  \varepsilon_{m,L} \leq (1 +\widetilde q^m)^L-1 
  = (1 +\widetilde q^m)^{\widetilde q^{-m} \widetilde q^m L}-1 
  \leq \exp(L \widetilde q^m) - 1, 
\end{equation*}
where we used $\sup_{x > 0} (1+x)^{1/x} \leq e$. 
Using the estimate $\exp(x) - 1 \leq e x$, which is valid for $x \in [0,1]$,
and assuming $\widetilde q^m L \leq 1$ (note that this holds for
$m \ge K (1 + \log L)$  with $K\geq 1/|\log \widetilde{q}|$), we obtain 
\begin{align*}
  \widehat q^{-m} \varepsilon_{m,L}
  &\leq e \widehat q^{-m} \widetilde q^m L
   = e (\widetilde q/\widehat q)^m L
   = \exp \bigl(\log L + \log e - m \log (\widehat q/\widetilde q)\bigr)\\
  &\leq \exp \bigl(\log L + \log e
             - K \log L \log(\widehat{q}/\widetilde{q})
             - K \log(\widehat{q}/\widetilde{q}) \bigr).
\end{align*}
Choosing $K := \max\{ m_0, m_1, 1/\log(\widehat{q}/\widetilde{q}),
1/|\log \widetilde{q}| \}$ completes the proof.
\end{proof}

\subsection{Multidimensional nested interpolation}

Using Theorem~\ref{th:iterated-stability-neu}, we can investigate the
stability and approximation properties of the multidimensional directional
interpolation operator.
We recall (\ref{eq:reinterpolation_reduced}), i.e.,
\begin{equation*}
  \mathfrak{I}_{\tau_L,c_L}\circ\cdots\circ\mathfrak{I}_{\tau_0,c_0}
  = (\mathfrak{I}_{L,1}\circ\cdots\circ\mathfrak{I}_{0,1})\otimes\cdots
    \otimes(\mathfrak{I}_{L,n}\circ\cdots\circ\mathfrak{I}_{0,n}).
\end{equation*}

%
%
\begin{theorem}[Nested directional interpolation]
\label{thm:nested_directional}
Let $\bar q\in(0,1)$ and assume that 
the nested sequence $\tau_L \subset \tau_{L-1} \subset \cdots \subset \tau_0$ of boxes of the form 
(\ref{eq:boxes}) satisfies 
\begin{align}\label{eq:contraction_boxes}
  \frac{b_{\ell,\iota}-a_{\ell,\iota}}
       {b_{\ell-1,\iota}-a_{\ell-1,\iota}} &\leq \bar q &
  &\text{ for all } \ell\in[1:L],\ \iota\in[1:n]. 
\end{align}
Assume that a sequence $(c_\ell)_{\ell=0}^L \subset \bbbr^n$ satisfies, for some $\gamma\in\bbbr_{>0}$, 
\begin{align}\label{eq:directions_boxes}
  \kappa \diam(\tau_{\ell-1}) \|c_{\ell-1}-c_\ell\| &\leq \gamma &
  &\text{ for all } \ell\in[1:L]. 
\end{align}
Let $\widehat{q}\in(\bar{q},1)$.
Then there is $K$ that depends solely on $\gamma$, $\bar q$, 
the chosen $\widehat{q}$, as well as $C_\Lambda$, $\lambda$ of (\ref{eq:polynomial-growth-Lambda})
such that for all $m\geq K (1+\log L)$ we have for the operator of 
(\ref{eq:reinterpolation_reduced}) the estimate 
\begin{equation*}
  \|\mathfrak{I}_{\tau_L,c_L}\circ\cdots\circ\mathfrak{I}_{\tau_0,c_0}
  \|_{\infty,\tau_L\leftarrow\tau_0}
  \leq \Lambda_m^n (1+\widehat{q}^m)^n.
\end{equation*}
\end{theorem}
\begin{proof}
We can apply Theorem~\ref{th:iterated-stability-neu} to get
\begin{align*}
  \|\mathfrak{I}_{\tau_L,c_L}\circ\cdots\circ\mathfrak{I}_{\tau_0,c_0}
     \|_{\infty,\tau_L\leftarrow\tau_0}
  &\leq \prod_{\iota=1}^n
     \|\mathfrak{I}_{L,\iota}\circ\cdots\circ\mathfrak{I}_{0,\iota}
     \|_{\infty,[a_{L,\iota},b_{L,\iota}]\leftarrow[a_{0,\iota},b_{0,\iota}]}\\
  &\leq \Lambda_m^n (1+\varepsilon_{m,L})^n
   \leq \Lambda_m^n (1+\widehat{q}^m)^n
\end{align*}
for all $m\geq K (1 + \log L)$.
\end{proof}

%
%
\begin{corollary}[Re-interpolated Lagrange polynomials]
\label{co:reinterpolated_lagrange}
Assume that the inequalities
(\ref{eq:contraction_boxes}) (\ref{eq:directions_boxes}) hold with $\bar q \in (0,1)$ and $\gamma \in \bbbr_{>0}$.
Fix $\widehat q \in (q,1)$. Then there is $K > 0$ depending only on $\bar{q}$, $\widehat q$, $\gamma$, 
and the constants $C_\Lambda$, $\lambda$ of (\ref{eq:polynomial-growth-Lambda})
such that for
all $m\geq K (1+\log L)$ we have for the functions 
$\widetilde{L}_{\tau_0 c_0,\nu} = 
  \mathfrak{I}_{\tau_L,c_L}\circ\cdots\circ\mathfrak{I}_{\tau_0,c_0} L_{\tau_0 c_0, \nu}$
introduced in (\ref{eq:Ltilde}) 
\begin{align*}
  \|\widetilde{L}_{\tau_0 c_0,\nu} - L_{\tau_0 c_0,\nu}
  \|_{\infty,\tau_L}
  &\leq \widehat{q}^m \|L_{\tau_0 c_0,\nu}\|_{\infty,\tau_0} &
  &\text{ for all } \nu\in M.
\end{align*}
\end{corollary}
\begin{proof}
We have
\begin{equation*}
  L_{\tau_0 c_0,\nu}
  = (\exp(\bi \kappa c_{0,1} \punkt) L_{[a_{0,1},b_{0,1}],\nu_1})\otimes\cdots
    \otimes(\exp(\bi \kappa c_{0,n} \punkt) L_{[a_{0,n},b_{0,n}],\nu_n})
\end{equation*}
with the transformed univariate Lagrange polynomials $L_{[a,b],\mu}$
defined in (\ref{eq:xi_L_trans_def}).
We use a telescoping sum to handle the dimensions $\iota\in[1:n]$
and obtain from 
(\ref{eq:thm:nested-interpolation-100}) and 
(\ref{eq:thm:nested-interpolation-200})
$$
  \|\widetilde{L}_{\tau_0 c_0,\nu} - L_{\tau_0 c_0,\nu}
  \|_{\infty,\tau_L}
  \leq n (1 + \varepsilon_{m,L})^{n-1} \varepsilon_{m,L} \|L_{\tau_0 c_0,\nu}\|_{\infty,\tau_0}
  \quad \text{ for all } \nu\in M.
$$
The result now follows in view of (\ref{eq:lemma:iterated-stability-neu-20}). 
\end{proof}

%
%
\begin{remark}
As observed in connection with (\ref{eq:c=0}), we choose $c_\ell = 0$ if
the boxes are sufficiently small relative to the wavelength  $2\pi/\kappa$.
In this case, the functions $\widetilde L_{\tau c,\nu}$ are standard
polynomials and the re-interpolation does not incur any error.
In other words: If (\ref{eq:c=0}) holds, then $L = \mathcal{O}(\log \kappa)$ so
that the condition $m \ge K (1 + \log L)$ reduces to
$m \ge K' \log (\log \kappa)$ for some $K'$. 
\end{remark}
%
%
%
\section{Numerical experiments}
\label{sec:numerics}

In order to investigate how accurately our theoretical results
predict the convergence of an actual implementation of our
nested interpolation scheme, we have implemented a ``pure'' version
of the algorithm outlined in Section~\ref{sec:algorithm}, i.e., a
version that does not use adaptive techniques to improve the compression rate.
While we acknowledge that for practical applications an algebraic
recompression scheme \cite{MESCDA12,BEKUVE15,BO04,BO10} is crucial,
we have chosen this approach to avoid pitfalls like unrealistically
low errors due to full rank ``approximations''.

We satisfy the admissibility condition (\ref{eq:adm_directional})
by assigning each level $\ell$ of the cluster tree a set $\mathcal{D}_\ell$
of directions constructed as follows:
we denote the maximal diameter of all clusters on level $\ell$ by
$\delta_\ell$ and split the surface of the cube $[-1,1]^3$ into squares with
diameter $\leq 2 \eta_1 / (\kappa \delta_\ell)$.
The midpoints $\tilde c$ of these squares are then projected by
$c:=\tilde c / \|\tilde c\|$ to the unit sphere.
By construction, each point on the cube's surface has a distance of
less than $\eta_1 / (\kappa \delta_\ell)$ to one of the midpoints, and
the projection cannot increase this distance.

We use the unit sphere as the surface $\Gamma$ for our test,
approximated by a triangulation constructed by regularly subdividing
the faces of a double pyramid into smaller triangles and projecting the
resulting vertices to the unit sphere.
We use meshes with $N\in\{4608, 8192, 18432, 32768, 73728, 131072\}$
triangles.

%
%
\begin{table}
  \begin{equation*}
    \begin{array}{r|r|r|llllll|l}
    \hline
    & & & \multicolumn{6}{|c}{\eta_1=10,\ \eta_2=1} & \\
    \hline
    N & \kappa & \|G\|_2 & m=2 & m=3 & m=4 & m=5 & m=6 & m=7 & \\   
    \hline
    9\times 2^{9\phantom{0}} & 6 & 8.2_{-4}
    & 1.6_{-6} & 5.1_{-8} & 3.2_{-10} & 1.4_{-11} & 4.4_{-13} & 2.0_{-14} & 0.05\\
    2\times 2^{12} & 8 & 3.7_{-4}
    & 1.4_{-6} & 1.4_{-7} & 5.7_{-9} & 4.6_{-10} & 2.9_{-11} & 1.8_{-12} & 0.06\\
    9\times 2^{11} & 12 & 1.3_{-4}
    & 9.9_{-7} & 1.4_{-7} & 1.0_{-8} & 1.0_{-9} & 9.0_{-11} & 7.1_{-12} & 0.08\\
    2\times 2^{14} & 16 & 5.8_{-5}
    & 7.0_{-7} & 7.0_{-8} & 6.1_{-9} & 7.0_{-10} & 7.5_{-11} & 7.0_{-12} & 0.09\\
    9\times 2^{13} & 24 & 2.0_{-5}
    & 2.0_{-7} & 2.0_{-8} & 2.4_{-9} & 2.8_{-10} & 2.9_{-11} & & (0.10)\\
    2\times 2^{16} & 32 & 9.2_{-6}
    & 1.1_{-7} & 1.4_{-8} & 1.8_{-9} & & & & (0.13)\\
    \hline
    & & & \multicolumn{6}{|c}{\eta_1=10,\ \eta_2=2} & \\
    \hline
    N & \kappa & \|G\|_2 & m=2 & m=3 & m=4 & m=5 & m=6 & m=7 & \\   
    \hline
    9\times 2^{9\phantom{0}} & 6 & 8.2_{-4}
    & 5.1_{-6} & 6.6_{-7} & 5.5_{-8} & 5.4_{-9} & 4.7_{-10} & 3.6_{-11} & 0.08\\
    2\times 2^{12} & 8 & 3.7_{-4}
    & 3.7_{-6} & 3.9_{-7} & 2.5_{-8} & 2.7_{-9} & 2.7_{-10} & 2.6_{-11} & 0.10\\
    9\times 2^{11} & 12 & 1.3_{-4}
    & 2.4_{-6} & 3.0_{-7} & 3.0_{-8} & 3.2_{-9} & 3.1_{-10} & 2.5_{-11} & 0.08\\
    2\times 2^{14} & 16 & 5.8_{-5}
    & 1.2_{-6} & 1.6_{-7} & 2.3_{-8} & 3.0_{-9} & 3.6_{-10} & 3.7_{-11} & 0.10\\
    9\times 2^{13} & 24 & 2.0_{-5}
    & 2.8_{-7} & 4.0_{-8} & 5.8_{-9} & 8.0_{-10} & 9.8_{-11} & & (0.12)\\
    2\times 2^{16} & 32 & 9.2_{-6}
    & 1.5_{-7} & 2.2_{-8} & 3.8_{-9} & 6.4_{-10} & & & (0.17)\\
    \hline
    \end{array}
  \end{equation*}
  \caption{Approximation errors $\|G-\widetilde{G}\|_2$ for
           the unit sphere}
  \label{ta:convergence}
\end{table}

%
%
\begin{figure}
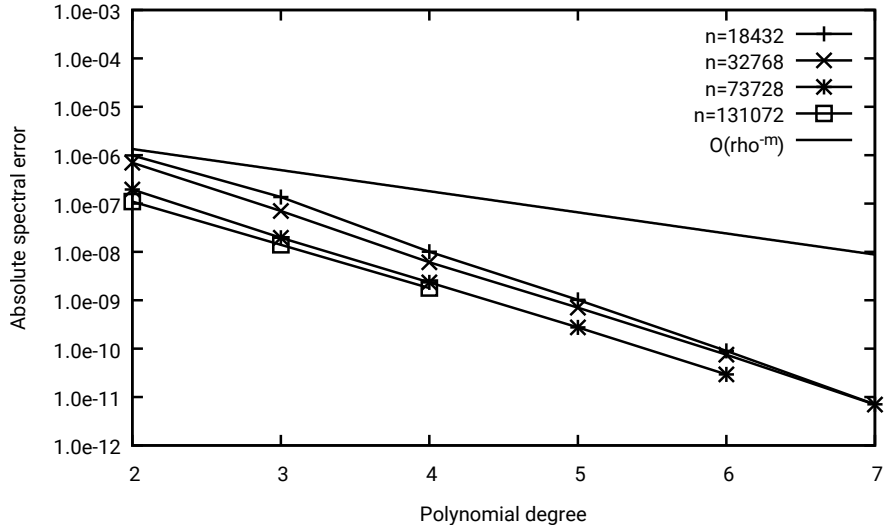

  \pgfdeclareimage[width=12cm]{convergence_eta2_1}{fi_convergence_eta2_1}

  \begin{center}
  \pgfuseimage{convergence_eta2_1}
  \end{center}
  \caption{Approximation errors for different meshes and different
           interpolation degrees}
  \label{fi:convergence}
\end{figure}

The cluster tree is set up by geometrical bisection.
The algorithm stops subdividing clusters $\tau$ as soon as the
corresponding index set $\hat\tau$ contains not more than $32$ indices
for $m=2$, $48$ indices for $m=3$, and $64$ indices for $m>3$.

The wave number $\kappa$ is chosen to provide us with a
high-frequency problem: we have $\kappa h \approx 0.6$, where
$h$ denotes the maximal mesh width, i.e., we have approximately
ten mesh elements per wavelength.

The approximation $\widetilde{G}$ constructed by our algorithm is
compared to the original matrix $G$, and the spectral norm
$\|G-\widetilde{G}\|_2$ of the error is approximated by $20$ steps
of the power iteration applied to the matrix
$(G-\widetilde{G})^* (G-\widetilde{G})$.
The norm $\|G\|_2$ is approximated in the same way.

Table~\ref{ta:convergence} summarizes our results:
the rows correspond to the different meshes, while the columns
give the spectral error estimates for different interpolation
degrees $m\in[2:7]$.
Missing numbers correspond to experiments that did not fit
into our machine's memory.

The last column of Table~\ref{ta:convergence} gives the
quotient of the last two computed errors, and we expect it to be a
good approximation of the asymptotic convergence rate.

We investigate two choices for the admissibility parameters
$\eta_1$ and $\eta_2$:
for $\eta_1=10$, $\eta_2=1$, our theory predicts an asymptotic
convergence rate of $1/(\sqrt{2}+1) \approx 0.41$.
We can see that the convergence rates in Table~\ref{ta:convergence}
are significantly smaller than this theoretical bound.
This is also illustrated in Figure~\ref{fi:convergence} showing the
measured errors for the four finest meshes and the slope predicted
by our analysis.

For the second choice $\eta_1=10$, $\eta_2=2$, we only expect a
convergence rate of $1/(\sqrt{5/4}+1/2) \approx 0.52$.
Once again, the measured rates are significantly better than
predicted.

Table~\ref{ta:convergence_cube} lists the results for the surface
of the cube $[-1,1]^3$ instead of the sphere, discretized with
$N\in\{6912,12288,27648,49152,110592\}$ triangles.
The convergence rates are very similar, illustrating that the
directional interpolation does not rely on the smoothness of
the surface $\Gamma$.

%
%
\begin{table}
  \begin{equation*}
    \begin{array}{r|r|r|llllll|l}
    \hline
    & & \multicolumn{6}{|c}{\eta_1=10,\ \eta_2=1} & \\
    \hline
    N & \kappa & \|G\|_2 & m=2 & m=3 & m=4 & m=5 & m=6 & m=7 & \\   
    \hline
    27\times 2^{8\phantom{0}} & 6 & 1.0_{-3}
    & 5.8_{-6} & 4.3_{-7} & 2.4_{-8} & 1.4_{-9} & 8.2_{-11} & 4.7_{-12} & 0.06\\
    3\times 2^{12} & 8 & 4.6_{-4}
    & 5.1_{-6} & 5.0_{-7} & 5.8_{-8} & 4.2_{-9} & 3.0_{-10} & 1.7_{-11} & 0.06\\
    27\times 2^{10} & 12 & 1.5_{-4}
    & 3.2_{-6} & 4.4_{-7} & 5.7_{-8} & 6.7_{-9} & 6.4_{-10} & 5.2_{-11} & 0.08\\
    3\times 2^{14} & 16 & 7.1_{-5}
    & 2.0_{-6} & 4.2_{-7} & 6.8_{-8} & 1.0_{-8} & 1.4_{-9} & 1.6_{-10} & 0.11\\
    27\times 2^{12} & 24 & 2.5_{-5}
    & 8.3_{-7} & 1.3_{-7} & 1.9_{-8} & & & & (0.15)\\
    \hline
    & & \multicolumn{6}{|c}{\eta_1=10,\ \eta_2=2} & \\
    \hline
    N & \kappa & \|G\|_2 & m=2 & m=3 & m=4 & m=5 & m=6 & m=7 & \\   
    \hline
    27\times 2^{8\phantom{0}} & 6 & 1.0_{-3}
    & 5.8_{-6} & 4.3_{-7} & 2.4_{-8} & 1.4_{-9} & 8.2_{-11} & 4.7_{-12} & 0.06\\
    3\times 2^{12} & 8 & 4.6_{-4}
    & 1.4_{-5} & 2.7_{-6} & 2.8_{-7} & 3.7_{-8} & 4.2_{-9} & 4.4_{-10} & 0.10\\
    27\times 2^{10} & 12 & 1.5_{-4}
    & 5.3_{-6} & 9.1_{-7} & 1.2_{-7} & 1.6_{-8} & 1.5_{-9} & 1.8_{-10} & 0.12\\
    3\times 2^{14} & 16 & 7.1_{-5}
    & 3.9_{-6} & 9.5_{-7} & 1.6_{-7} & 2.6_{-8} & 3.6_{-9} & 4.2_{-10} & 0.12\\
    27\times 2^{12} & 24 & 2.5_{-5}
    & 1.2_{-6} & 1.8_{-7} & 2.5_{-8} & 4.0_{-9} & & & (0.16)\\
    \hline
    \end{array}
  \end{equation*}
  \caption{Approximation errors $\|G-\widetilde{G}\|_2$ for the
           surface of the cube $[-1,1]^3$}
  \label{ta:convergence_cube}
\end{table}

%
%

\bibliographystyle{plain}
\bibliography{hmatrix}

%
%

\appendix

\section{\texorpdfstring{$\mathcal{DH}^2$-matrices}
                        {DH2-matrices}}

Our results can be used to prove convergence of various efficient
algorithms for the Helmholtz boundary integral equation.
We now discuss a straightforward approach that leads to what we
call \emph{directional $\mathcal{H}^2$-matrices}, or short
$\mathcal{DH}^2$-matrices.
These matrices are a generalization of the
\emph{$\mathcal{H}^2$-matrix representation} \cite{HAKHSA00,BOHA02,BO10}.

Our definition of $\mathcal{DH}^2$-matrices is not identical to the
one used in \cite{BEKUVE15}, since we do not switch to an
$\mathcal{H}$-matrix representation for the low-frequency case, but
use $\mathcal{H}^2$-matrix representations for all blocks.

%
%
\begin{definition}[Cluster tree]
\label{de:cluster}
Let $\mathcal{T}$ be a labeled tree such that the label $\hat t$
of each node $t\in\mathcal{T}$ is a subset of the index set $\Idx$.
We call $\mathcal{T}$ a \emph{cluster tree} for $\Idx$ if
\begin{itemize}
  \item the root $r\in\mathcal{T}$ is assigned the label $\hat r=\Idx$,
  \item the index sets of siblings are disjoint, i.e.,
    \begin{align*}
      t_1\neq t_2 &\Longrightarrow \hat t_1\cap\hat t_2=\emptyset &
      &\text{ for all } t\in\mathcal{T},\ t_1,t_2\in\sons(t),
         \text{ and}
    \end{align*}
  \item the index sets of a cluster's sons are a partition of
    their father's index set, i.e.,
    \begin{align*}
      \hat t &= \bigcup_{t'\in\sons(t)} \hat t' &
      &\text{ for all } t\in\mathcal{T}
          \text{ with } \sons(t)\neq\emptyset.
    \end{align*}
\end{itemize}
A cluster tree for $\Idx$ is usually denoted by $\ctI$.
Its nodes are called \emph{clusters}.
\end{definition}

A cluster tree $\ctI$ can be split into levels:
we let $\ctIl{0}$ be the set containing only the root of $\ctI$
and define
\begin{align*}
  \ctIl{\ell} &:= \{ t'\in\ctI\ :\ t'\in\sons(t) \text{ for a }
                       t\in\ctIl{\ell-1} \} &
  &\text{ for all } \ell\in\bbbn.
\end{align*}
For each cluster $t\in\ctI$, there is exactly one $\ell\in\bbbn_0$
such that $t\in\ctIl{\ell}$.
We call this the \emph{level number} of $t$ and denote it by $\level(t)=\ell$.
The maximal level
\begin{equation*}
  p_\Idx := \max\{ \level(t)\ :\ t\in\ctI \}
\end{equation*}
is called the \emph{depth} of the cluster tree.

Pairs of clusters $(t,s)$ correspond to subsets $\hat t\times\hat s$
of $\Idx\times\Idx$, and by extension to submatrices of
$G\in\bbbc^{\Idx\times\Idx}$.
These pairs inherit the hierarchical structure provided by the
cluster tree.

%
%
\begin{definition}[Block tree]
\label{de:block}
Let $\mathcal{T}$ be a labeled tree, and let $\ctI$ be a cluster tree
for the index set $\Idx$ with root $r_\Idx$.
We call $\mathcal{T}$ a \emph{block tree} for $\ctI$ if
\begin{itemize}
  \item for each $b\in\mathcal{T}$ there are $t,s\in\ctI$
        such that $b=(t,s)$,
  \item the root $r\in\mathcal{T}$ satisfies $r=(r_\Idx,r_\Idx)$,
  \item the label of $b=(t,s)\in\mathcal{T}$ is given by
        $\hat b=\hat t\times\hat s$, and
  \item for each $b=(t,s)\in\mathcal{T}$ we have
    \begin{equation*}
      \sons(b)\neq\emptyset \Longrightarrow
      \sons(b) = \sons(t)\times\sons(s).
    \end{equation*}
\end{itemize}
A block tree for $\ctI$ is usually denoted by $\ctII$.
Its nodes are called \emph{blocks}.
\end{definition}

In the following, we assume that a cluster tree $\ctI$ for the
index set $\Idx$ and a block tree $\ctII$ for $\ctI$ are given.

We have to identify submatrices, corresponding to blocks, that can
be approximated efficiently.
Considering the form (\ref{eq:matrix}) of the matrix entries, we
require the approximation $\tilde g_{ts}$ of the kernel function $g$
to be valid in the entire support of the basis functions
$\varphi_i$ and $\varphi_j$ for $i\in\hat t$ and $j\in\hat s$.

%
%
\begin{definition}[Bounding box]
Let $t\in\ctI$ be a cluster.
An axis-parallel box $\tau\subseteq\bbbr^3$ is called a
\emph{bounding box} for $t$ if
\begin{align*}
  \supp(\varphi_i) &\subseteq \tau &
  &\text{ for all } i\in\hat t.
\end{align*}
\end{definition}

In practice we can construct bounding boxes of minimal size by
a simple and fast recursive algorithm \cite[Example~2.2]{BOGRHA03}.

Our approximation scheme (\ref{eq:directional_approximation}) requires
a direction for the plane wave.
In order to obtain the optimal order of complexity, we fix a finite set of
directions for each level of the cluster tree and introduce a
connection between the directions for a cluster $t$ and the directions
for its sons $t'\in\sons(t)$.

%
%
\begin{definition}[Hierarchical directions]
\label{de:directions}
A family $(\mathcal{D}_\ell)_{\ell=0}^\infty$ of finite subsets
of $\bbbr^3$ is called a \emph{family of hierarchical directions} if
\begin{align*}
  \|c\| = 1 &\vee c = 0  &
  &\text{ for all } c\in\mathcal{D}_\ell,\ \ell\in\bbbn_0.
\end{align*}
A family $(\sd{\ell})_{\ell=0}^\infty$ of mappings
$\sd{\ell}:\mathcal{D}_\ell\to\mathcal{D}_{\ell+1}$ is 
called a \emph{family of compatible son mappings} if
\begin{align*}
  \|c-\sd{\ell}(c)\| &\leq \|c-\tilde c\| &
  &\text{ for all } c\in\mathcal{D}_\ell,\ \tilde c\in\mathcal{D}_{\ell+1},
                    \ \ell\in\bbbn_0.
\end{align*}
Given a cluster tree $\ctI$, a family of hierarchical directions, 
and a family of compatible son mappings, we write
\begin{align*}
  \mathcal{D}_t &:= \mathcal{D}_{\level(t)}, &
  \sd{t}(c) &:= \sd{\level(t)}(c) &
  &\text{ for all } t\in\ctI,\ c\in\mathcal{D}_{\level(t)}.
\end{align*}
\end{definition}

%
%
\begin{remark}
The ``direction'' $c=0$ is included in Definition~\ref{de:directions}
in order to include the low-frequency case in our scheme in a convenient way,
cf. (\ref{eq:c=0}).
\end{remark}

%
%
\begin{remark}[Implementation]
In practice, we only have to define $\mathcal{D}_\ell$ for
$\ell\leq p_\Idx$ and $\sd{\ell}$ for $\ell<p_\Idx$.
Our definition admits infinite levels only to avoid special cases.
\end{remark}

In the following, we fix a cluster tree $\ctI$, a family
$(\mathcal{D}_\ell)_{\ell=0}^\infty$ of hierarchical directions and
a family $(\sd{\ell})_{\ell=0}^\infty$ of compatible son mappings.

Assume that a block $b=(t,s)\in\ctII$ and a direction
$c=c_b\in\mathcal{D}_t=\mathcal{D}_s$ is given.
According to (\ref{eq:VSW}), replacing $g$ in (\ref{eq:matrix}) with the
directional approximation
\begin{align*}
  \tilde g_{ts}(x,y)
  &= \sum_{\nu\in M} \sum_{\mu\in M} g_c(\xi_{t,\nu}, \xi_{s,\nu})
         L_{tc,\nu}(x) \overline{L_{sc,\mu}(y)} &
  &\text{ for all } x\in \tau,\ y\in \sigma
\end{align*}
yields a low-rank factorization of the form
\begin{equation}\label{eq:matrix_apx}
  G|_{\hat t\times\hat s} \approx V_{tc} S_b V_{sc}^*.
\end{equation}
The directional re-interpolation of
the Lagrange polynomials described in (\ref{eq:lagrange_reinterpolation})
leads to the nested representation
\begin{equation}\label{eq:nested_apx}
  V_{tc}|_{\hat t'\times k} \approx V_{t'c'} E_{t'c}
\end{equation}
of the matrices $V_{tc}$.
This approximation brings about a complexity reduction since only the small
matrices $E_{t'c} \in \bbbc^{k \times k}$ need to be stored instead of 
$V_{tc}\in\bbbc^{\hat t\times k}$.  
The notation $E_{t'c}$ is well-defined since the father $t\in\ctI$ is
uniquely determined by $t'\in\sons(t)$ due to the tree structure
and the direction $c' = \sd{t}(c) \in\mathcal{D}_{t'}$ is uniquely determined
by $c\in\mathcal{D}_t$ due to our Definition~\ref{de:directions}.

%
%
\begin{definition}[Directional cluster basis]
Let $M$ be a finite index set, and let
$V=(V_{tc})_{t\in\ctI,c\in\mathcal{D}_t}$ be a family of matrices.
We call it a \emph{directional cluster basis} if
\begin{itemize}
  \item $V_{tc}\in\bbbc^{\hat t\times M}$ for all $t\in\ctI$
     and $c\in\mathcal{D}_t$, and
  \item there is a family
     $E=(E_{t'c})_{t\in\ctI,t'\in\sons(t),c\in\mathcal{D}_t}$
     such that
     \begin{align}\label{eq:nested-hinten}
       V_{tc}|_{\hat t'\times k} &= V_{t'c'} E_{t'c} &
       &\text{ for all } t\in\ctI,\ t'\in\sons(t),\ c\in\mathcal{D}_t,
                         \ c' = \sd{t}(c). 
     \end{align}
\end{itemize}
The elements of the family $E$ are called \emph{transfer matrices} for
the directional cluster basis $V$, and $k:=\#M$ is called its \emph{rank}.
\end{definition}

We can now define the class of matrices that is the subject of this
article:
we denote the \emph{leaves} of the block tree $\ctII$ by
\begin{equation*}
  \lfII := \{ b\in\ctII\ :\ \sons(b)=\emptyset \}.
\end{equation*}
The corresponding sets $\hat b\subseteq\Idx\times\Idx$ form a disjoint
partition of $\Idx\times\Idx$, so a matrix $G$ is uniquely determined
by the submatrices $G|_{\hat b}$ for $b\in\lfII$.
For most of these submatrices, we can find an approximation of the
form (\ref{eq:matrix_apx}).
These matrices are called \emph{admissible} and collected in a
subset
\begin{align*}
  \lfaII &:= \{ b\in\lfII\ :\ b \text{ is admissible} \}.\\
\intertext{The remaining blocks are called \emph{inadmissible} and
collected in the set}
  \lfiII &:= \lfII \setminus \lfaII.
\end{align*}
How to decide whether a block is admissible or not is the topic of 
Section~\ref{su:admissibility}.

%
%
\begin{definition}[Directional $\mathcal{H}^2$-matrix]
Let $V$ and $W$ be directional cluster bases for $\ctI$.
Let $G\in\bbbc^{\Idx\times\Idx}$ be a matrix.
We call it a \emph{directional $\mathcal{H}^2$-matrix} (or simply a
\emph{$\mathcal{DH}^2$-matrix}) if there are families
$S=(S_b)_{b\in\lfaII}$ and $(c_b)_{b\in\lfaII}$ such that
\begin{itemize}
  \item $S_b\in\bbbc^{k\times k}$ and $c_b\in\mathcal{D}_t=\mathcal{D}_s$
    for all $b=(t,s)\in\lfaII$, and
  \item $G|_{\hat t\times\hat s} = V_{tc} S_b W_{sc}^*$ with $c=c_b$ for all
    $b=(t,s)\in\lfaII$.
\end{itemize}
The elements of the family $S$ are called \emph{coupling matrices},
and $c_b$ is called the \emph{block direction} for $b\in\ctII$.
The cluster bases $V$ and $W$ are called the \emph{row cluster basis}
and \emph{column cluster basis}, respectively.

A \emph{$\mathcal{DH}^2$-matrix representation} of a
matrix $G$ consists of $V$, $W$, $S$ and the family
$(G|_{\hat b})_{b\in\lfiII}$ of \emph{nearfield matrices} corresponding
to the inadmissible leaves of $\ctII$.
\end{definition}

Let $G$ be a $\mathcal{DH}^2$-matrix for the directional cluster
bases $V$ and $W$, and let $x\in\bbbc^\Idx$.
We denote the corresponding cluster tree by $\ctI$ and the
corresponding block tree by $\ctII$ with admissible leaves
$\lfaII$.
For an efficient evaluation of the matrix-vector product $y=Gx$.
we follow the familiar approach of fast multipole and
$\mathcal{H}^2$-matrix techniques:
since the submatrices are factorized into three
terms
\begin{align*}
  G|_{\hat t\times\hat s} &= V_{tc} S_b W_{sc}^* &
  &\text{ for all } b=(t,s)\in\lfaII,
\end{align*}
the algorithm is split into three phases:
in the first phase, called the \emph{forward transformation},
we multiply by $W_{sc}^*$ and compute
\begin{subequations}
\begin{align}
  \widehat{x}_{sc} &= W_{sc}^* x|_{\hat s} &
  &\text{ for all } s\in\ctI,\ c\in\mathcal{D}_s;\label{eq:mvm1}
\intertext{in the second phase, the \emph{coupling step},
we multiply these coefficient vectors by the coupling matrices $S_b$
and obtain}
  \widehat{y}_{tc}
  &:= \sum_{\substack{b=(t,s)\in\lfaII\\ c=c_b}} S_b \widehat{x}_{sc} &
  &\text{ for all } t\in\ctI,\ c\in\mathcal{D}_t;\label{eq:mvm2}
\intertext{and in the final phase, the \emph{backward transformation},
we multiply by $V_{tc}$ to get the result}
  y_i &= \sum_{\substack{t\in\ctI,\ c\in\mathcal{D}_t\\ i\in\hat t}}
           (V_{tc} \widehat{y}_{tc})_i &
  &\text{ for all } i\in\Idx.\label{eq:mvm3}
\end{align}
\end{subequations}
The first and third phase can be handled efficiently by using
the transfer matrices $E_{t'c}$:
let $s\in\ctI$ with $\sons(s)\neq\emptyset$, and let $c\in\mathcal{D}_s$.
Due to the structure of the cluster tree, the set
$\{\hat s'\ :\ s'\in\sons(s)\}$ is a disjoint partition of the index
set $\hat s$.
Combined with (\ref{eq:nested-hinten}), this implies
\begin{equation*}
  W_{sc}^* x|_{\hat s}
   = \sum_{s'\in\sons(s)} (W_{sc}|_{\hat s'\times k})^* x|_{\hat s'}
   = \sum_{s'\in\sons(s)} E_{s'c}^* V_{s'c'}^* x|_{\hat s'}
   = \sum_{s'\in\sons(s)} E_{s'c}^* \widehat{x}_{s'c'},
\end{equation*}
and we can prepare \emph{all} coefficient vectors $\widehat{x}_{sc}$
by the simple recursion given on the left of Figure~\ref{fi:forward_backward}.
By similar arguments we find that the third phase can also be handled
by the recursion given on the right of Figure~\ref{fi:forward_backward}.

%
%
\begin{figure}
  \begin{minipage}[t]{0.47\textwidth}
  \begin{tabbing}
    \textbf{ procedure} forward($s$, $x$, \textbf{ var} $\widehat{x}$);\\
    \textbf{ if} $\sons(s)=\emptyset$ \textbf{ then}\\
    \quad\= \textbf{ for} $c\in\mathcal{D}_s$ \textbf{ do}
               $\widehat{x}_{sc} \gets W_{sc}^* x|_{\hat s}$\\
    \textbf{ else begin}\\
    \> \textbf{ for} $s'\in\sons(s)$ \textbf{ do}
               forward($s'$, $x$, $\widehat{x}$);\\
    \> \textbf{ for} $c\in\mathcal{D}_s$ \textbf{ do begin}\\
    \> \quad\= $\widehat{x}_{sc} \gets 0$;\\
    \> \> \textbf{ for} $s'\in\sons(s)$ \textbf{ do}\\
    \> \> \quad\= $\widehat{x}_{sc} \gets \widehat{x}_{sc}
                  + E_{s'c}^* \widehat{x}_{s'c'}$\\
    \> \textbf{ end}\\
    \textbf{ end}
  \end{tabbing}
  \end{minipage}%
  \hfill%
  \begin{minipage}[t]{0.47\textwidth}
  \begin{tabbing}
    \textbf{ procedure} backward($t$, \textbf{ var} $\widehat{y}$, $y$);\\
    \textbf{ if} $\sons(t)=\emptyset$ \textbf{ then}\\
    \quad\= \textbf{ for} $c\in\mathcal{D}_s$ \textbf{ do}
               $y|_{\hat t} \gets y|_{\hat t} + V_{tc} \widehat{y}_{tc}$\\
    \textbf{ else begin}\\
    \> \textbf{ for} $c\in\mathcal{D}_t$ \textbf{ do}\\
    \> \quad\= \textbf{ for} $t'\in\sons(t)$ \textbf{ do}\\
    \> \> \quad\= $\widehat{y}_{t'c'} \gets \widehat{y}_{t'c'}
                  + E_{t'c} \widehat{y}_{tc}$;\\
    \> \textbf{ for} $t'\in\sons(t)$ \textbf{ do}
               backward($t'$, $\widehat{y}$, $y$)\\
    \textbf{ end}
  \end{tabbing}
  \end{minipage}

  \caption{Fast forward and backward transformation}
  \label{fi:forward_backward}
\end{figure}

The submatrices corresponding to inadmissible leaves $b=(t,s)\in\lfiII$
are stored as standard arrays and can be evaluated accordingly.

We see that the algorithms use each of the matrices of the
$\mathcal{DH}^2$-matrix representation exactly once, so the bound
provided by Remark~\ref{re:directions} for the storage requirements
yields an $\mathcal{O}(N k + \kappa^2 k^2 \log N)$ complexity 
for the matrix-vector multiplication.

\end{document}